\documentclass[twoside, openany ]{amsart}
\usepackage{amsfonts,amsmath, amssymb}
\usepackage[linktocpage=true, colorlinks=true, linkcolor=RoyalBlue, filecolor=RoyalBlue,urlcolor=RoyalBlue,citecolor=RoyalBlue,hypertexnames=false]{hyperref}
\usepackage{aliascnt}
\usepackage{varioref}
\usepackage{graphicx}
\usepackage{aligned-overset}
\usepackage{float}
\usepackage{srcltx}
\usepackage[all]{xy}
\usepackage{bbm}
\usepackage[T2A, T1]{fontenc}
\usepackage[utf8]{inputenc}	
\usepackage{CJKutf8}
\usepackage{adjustbox}
\usepackage[dvipsnames]{xcolor}
\usepackage[normalem]{ulem}
\usepackage{bm}
\usepackage{pifont}
\usepackage{latexsym}
\usepackage{stmaryrd}
\usepackage[2emode]{psfrag}
\usepackage{yhmath}
\usepackage{array}
\usepackage{dsfont}
\usepackage{mathrsfs}% for \mathscr
\usepackage{tikz-cd}% diagrams
\usepackage{comment}
\usepackage{subcaption}
\usetikzlibrary { decorations.pathmorphing, decorations.pathreplacing, decorations.shapes }
\usetikzlibrary{graphs}
\usepackage{maybemath}
\usepackage[capitalize,noabbrev]{cleveref}
\crefname{section}{\S$\!\!$}{\S\S$\!\!$}
\Crefname{section}{Section}{Sections}
\crefname{subsection}{\S$\!\!$}{\S\S$\!\!$}
\Crefname{subsection}{Section}{Sections}
\crefname{subsubsection}{\S$\!\!$}{\S\S$\!\!$}
\Crefname{subsubsection}{Section}{Sections}
\DeclareSymbolFont{cyrillic}{T2A}{cmr}{m}{n}
\DeclareMathSymbol{\Zh}{\mathalpha}{cyrillic}{198}
\DeclarePairedDelimiter{\points}{<}{>}

\DeclareMathOperator*{\bigast}{\raisebox{-7pt}{\scalebox{2}{*}}\kern -2pt}

\DeclareMathOperator{\otb}{\bar{\otimes}}

\definecolor{Maroon}{rgb}{0.6 0 0}
\definecolor{Prussian}{rgb}{0.05 0 0.6}
\definecolor{Emerald}{rgb}{0 0.5 0.1}
\newtheoremstyle{mytheorem}%
{10.0pt plus 2.0pt minus 2.0pt} % Space above
{10.0pt plus 2.0pt minus 2.0pt} % Space below
{\itshape} % Body font
{} % Indent amount
{\sc} % Theorem head font
{.} % Punctuation after theorem head
{ } % Space after theorem head
{} % Theorem head spec (can be left empty, meaning `normal')

\newtheoremstyle{mydefinition}%
{10.0pt plus 2.0pt minus 2.0pt} % Space above
{10.0pt plus 2.0pt minus 2.0pt} % Space below
{} % Body font
{} % Indent amount
{\sc} % Theorem head font
{.} % Punctuation after theorem head
{ } % Space after theorem head
{} % Theorem head spec (can be left empty, meaning `normal')

\newtheoremstyle{myexample}%
{10.0pt plus 2.0pt minus 2.0pt} % Space above
{10.0pt plus 2.0pt minus 2.0pt} % Space below
{\small} % Body font
{} % Indent amount
{\sc} % Theorem head font
{.} % Punctuation after theorem head
{ } % Space after theorem head
{} % Theorem head spec (can be left empty, meaning `normal')

\newtheoremstyle{myremark}%
{10.0pt plus 2.0pt minus 2.0pt} % Space above
{10.0pt plus 2.0pt minus 2.0pt} % Space below
{} % Body font
{} % Indent amount
{\itshape} % Theorem head font
{.} % Punctuation after theorem head
{ } % Space after theorem head
{} % Theorem head spec (can be left empty, meaning `normal')

\theoremstyle{mytheorem}
\newtheorem{theorem}{Theorem}[section]
\newaliascnt{lemma}{theorem}
\newtheorem{lemma}[lemma]{Lemma}
\aliascntresetthe{lemma}
\newaliascnt{conjecture}{theorem}

\aliascntresetthe{conjecture}
\newaliascnt{corollary}{theorem}
\newtheorem{corollary}[corollary]{Corollary}
\aliascntresetthe{corollary}
\newaliascnt{proposition}{theorem}
\newtheorem{proposition}[proposition]{Proposition}
\aliascntresetthe{proposition}

\theoremstyle{myremark}
\newaliascnt{remark}{theorem}
\newtheorem{remark}[remark]{Remark}
\aliascntresetthe{remark}
\newaliascnt{convention}{theorem}

\aliascntresetthe{convention}
\newaliascnt{notation}{theorem}

\aliascntresetthe{notation}

\theoremstyle{mydefinition}
\newaliascnt{definition}{theorem}
\newtheorem{definition}[definition]{Definition}
\aliascntresetthe{definition}
\newaliascnt{problem}{theorem}
\newtheorem{problem}[problem]{Problem}
\aliascntresetthe{problem}
\newaliascnt{question}{theorem}

\aliascntresetthe{question}

\theoremstyle{myexample}
\newaliascnt{example}{theorem}
\newtheorem{example}[example]{Example}
\aliascntresetthe{example}
\newaliascnt{counterexample}{theorem}

\aliascntresetthe{counterexample}

\newtheoremstyle{myzusatz}
 {10.0pt plus 2.0pt minus 2.0pt} % Space above
{10.0pt plus 2.0pt minus 2.0pt} % Space below
{\itshape} % Body font
{} % Indent amount
{\sc} % Theorem head font
{.} % Punctuation after theorem head
{ } % Space after theorem head
{\thmname{#1}\thmnumber{ #2}\thmnote{ #3}}

\theoremstyle{myzusatz}

\newaliascnt{zusatz}{theorem}

\aliascntresetthe{zusatz}

\definecolor{gray1}{gray}{0.8}
\definecolor{gray2}{gray}{0.6}
\definecolor{gray3}{gray}{0.4}
\definecolor{gray4}{gray}{0.2}
\allowdisplaybreaks
\newcommand{\id}{\mathrm{id}}

\newcommand{\Mm}{\mathscr{M}}

\newcommand{\Nn}{\mathscr{N}}

\newcommand{\Ii}{\mathscr{I}}

\newcommand{\Gg}{\mathscr{G}}

\newcommand{\Hh}{\mathscr{H}}
\newcommand{\Kk}{\mathscr{K}}

\newcommand{\Uu}{\mathscr{U}}

\newcommand{\ot}{\otimes}

\newcommand{\set}[1]{\lbrace #1\rbrace}

\def\Set{{\sf Set}}

\def\SKB{{\sf SKB}}
\def\QSKB{{\sf QSKB}}

\def\Gpd{{\sf Gpd}}
\def\Gpdconn{{\sf Gpd}_{\mathrm{conn}}}
\def\GpdSchur{{\sf Gpd}_{\mathrm{Schur}}}
\def\Gp{{\sf Gp}}

\def\Quiv{{\mathsf{Quiv}}}

\def\source{\mathfrak{s}}
\def\target{\mathfrak{t}}
\def\Source{\mathfrak{S}}
\def\Target{\mathfrak{T}}

\DeclareMathOperator{\Aut}{\mathrm{Aut}}

\newcommand{\Z}{\mathbb{Z}}

\newcommand{\blank}{\raisebox{-2pt}{\text{---}}}

\newcommand{\TT}{\mathrm{T}}

{
\left\{\begin{aligned}}
{\end{aligned}
\right.
}
\newcommand{\St}{\mathrm{St}}

\newcommand{\One}{\mathbbm{1}}
\DeclareMathOperator{\im}{\mathrm{im}}

\begin{document}
\hyphenation{qua-si-grou-po-id}
\hyphenation{qua-si-grou-po-ids}
\hyphenation{grou-po-id}
\hyphenation{grou-po-ids}
\hyphenation{oid-i-fi-ca-tion}
\hyphenation{a-quo-id}
\hyphenation{a-quo-ids}

\def\presub#1#2#3%
  {\mathop{}%
   \mathopen{\vphantom{#2}}_{#1}%
   \kern-\scriptspace%
   {#2}_{#3}}

\newcommand{\fibre}[2]{\presub{#1}{\times}{#2}}

\newcommand{\bicross}[2]{\presub{#1}{{\blacktriangleright\kern -3pt \blacktriangleleft}}{#2}}

\newcommand{\lcross}[2]{\presub{#1}{{\blacktriangleright\kern -3pt <}}{#2}}

\newcommand{\rcross}[2]{\presub{#1}{{> \kern -3pt \blacktriangleleft}}{#2}}

\newcommand{\twfibre}[2]{\presub{#1}{\,\protect\scalebox{0.65}[1]{\protect\ensuremath{\bowtie}}\,}{#2}}

\newcommand{\zappaszep}[2]{\presub{#1}{\,\bowtie\,}{#2}}

\newcommand{\lfibre}[2]{\presub{#1}{\rtimes}{#2}}

\newcommand{\rfibre}[2]{\presub{#1}{\ltimes}{#2}}

\newcommand{\twomapsright}[2]{\,\underset{#2}{\overset{#1}{\rightrightarrows}} \, }
%\title{Unearthing the lost Isomorphism Theorem}
\title{On quiver skew braces, their ideals and products}
\author{Davide Ferri}
\begin{abstract} \textit{Quiver skew braces} or \textit{skew bracoids} are equivalent to braided groupoids, that is, groupoids with a constraint of abelianity. They are the quiver-theoretic version of skew braces, an increasingly studied structure lying in the intersection of group and ring theory.
	
In this paper, we define ideals and quotients for quiver skew braces, with respect to two notions of morphisms. Following the track of a previous work of ours (2025), we define a classical semidirect product \textit{à la} Brown, and a categorical semidirect product \textit{à la} Bourn and Janelidze, for the category of quiver skew braces.
	
It is known that connected groupoids can be expressed as the datum of a group of loops and a set of vertices. We demonstrate how no such decomposition holds for quiver skew braces, which makes their theory richer than the theory of groupoids.
\end{abstract}
\address{%
\parbox[b]{0.9\linewidth}{University of Turin, Department of Mathematics `G.\@ Peano',\\ via
 Carlo Alberto 10, 10123 Torino, Italy.\\
 Vrije Universiteit Brussel, Department of Mathematics and Data Science,\\ Pleinlaan 2, 1050, Brussels, Belgium.}}
 \email{d.ferri@unito.it, Davide.Ferri@vub.be}
\maketitle
\tableofcontents 

\section*{Introduction}\label{chapt:intro}
This paper is conceived as a direct sequel of \cite{ferri2025splitandFIT}, and its investigation proceeds in the tracks of \cite{BaiGuoShengWang2025braideddynamicalgroups,ferri2024dynamical,sheng2024postgroupoidsquivertheoreticalsolutionsyangbaxter}. While in \cite{ferri2025splitandFIT} we studied the category of groupoids, here we consider the category of \textit{braided groupoids} (equivalently, what we call \textit{quiver skew braces}, and Sheng, Tang, and Zhu called \textit{skew bracoids}): we define their ideals and quotients, a `classical' semidirect product \textit{à la} Brown, and a categorical semidirect product in the sense of Bourn and Janelidze \cite{BournJanelidzeProtomodularityDescSemProd}. Moreover, using these products, we are able to characterise those quiver skew braces whose subgroupoid of loops is an ideal.

A \textit{braided group} is a group $G$ with a constraint of abelianity $r\colon G\times G\to G\times G$, called a \textit{braiding}, that satisfies properties analogous to the canonical flip $a\times b\mapsto b\times a $ on abelian groups \cite{LYZ}. Indeed, the canonical flip on $G$ is a braiding if and only if $G$ is abelian.

A set is a quiver with a single vertex. Following this idea, the philosophy of `oidification' is concerned with turning set-theoretic (i.e.\@, one-vertex) concepts into quiver-theoretic (multiple-vertices) ones. It is natural to ask for an oidification of braided groups. A \textit{braided groupoid} is a groupoid $\Gg$ with a morphism of quivers $r\colon \Gg\ot \Gg\to \Gg\ot \Gg$, also called a \textit{braiding}, that satisfies analogous axioms \cite{andruskiewitsch2005quiver}. Here $\ot$ denotes the tensor product of quivers over the same set of vertices $\Gg^0$ \cite{matsumotoshimizu}. Interestingly, though, this time there is no sensible notion of an `abelian groupoid' to generalise: if $a$ and $b$ are consecutive arrows, $b$ and $a$ need not be, so $ba$ is generally undefined. For the same reason, the canonical flip on a groupoid is, usually, not even a map of quivers. It is easy to get convinced that the canonical flip on $\Gg$ is a map of quivers if and only if $\Gg$ is a group, and that an abelian groupoid is forced to be just an abelian group. The notion of braiding is thus, in some sense, `more natural' than the notion of commutativity---at least in the context of quivers.

Braided groups are equivalent to \textit{skew braces}. A skew brace is a set with two group operations, satisfying a compatibility \cite{guarnieri2017skew,rump2007braces}. These objects have been classically perceived as having the flavour of a ring (radical rings yield, in fact, examples of skew braces \cite{rump2007braces}), hence they are often handled through ring-theoretic techniques.

Likewise, braided groupoids are equivalent to what we call \textit{quiver skew braces}, which are an oidification of skew braces. These are groupoids $\Gg$ with an extra group structure on each outgoing star $\St_\Gg(\lambda)$. They were introduced by Sheng, Tang, and Zhu \cite{sheng2024postgroupoidsquivertheoreticalsolutionsyangbaxter} with the name \textit{skew bracoids}. Since a distinct object named \textit{skew bracoid} already existed (see Martin-Lyons and Truman \cite{martin2024skew}), we prefer to adopt a different name here, for the sake of clarity.\footnote{New evidence suggests that Martin-Lyons and Truman's skew bracoids and Sheng, Tang, and Zhu's skew bracoids might be closely connected, which provides further motivation to dissipate any confusion around these two terms. It should be noted that I.\@ Martin-Lyons and P.\@ Truman picked the name `skew bracoid' while already  envisioning that it should have had some quiver-theoretic interpretation. The name `quiver skew brace' was proposed by the author together with I.\@ Colazzo.}

In this paper, we investigate the basic theory of braided groupoids, or equivalently quiver skew braces, in line with what was done for groupoids in \cite{BrownBookGroupoids,ferri2025splitandFIT,MetereMontoliSemidirectIntGpd} and many other places. In \cref{sec:QSKB} we recall the notions of groupoids, braided groupoids, and quiver skew braces. In \cref{sec:ideals} we define two notions of ideals for quiver skew braces: \textit{ideals}, corresponding to kernels of morphisms; and \textit{ideal bundles} corresponding to the kernels of strong morphisms. We then study the respective notions of quotients.

Two notions of `semidirect product' have been studied for groupoids. One, which we call Brown's \textit{classical} semidirect products \cite{BrownBookGroupoids}, makes more sense geometrically. The other one, described in \cite{ferri2025splitandFIT}, is Bourn and Janelidze's \textit{categorical} notion of semidirect product, corresponding to split epimorphisms. In \cref{sec:products}, we extend both notions to quiver skew braces.

As observed by Brown \cite{GroupsToGroupoidsBrown}, a (connected) groupoid $\Gg$ can be simply identified with the datum of a set and a group: the set of vertices $\Gg^0$, and any of the (isomorphic) isotropy groups $\Gg_\lambda$, $\lambda\in\Gg^0$. The theory of braided groupoids or quiver skew braces, however, is significantly more complex in this aspect. The datum `set of vertices and isotropy group' is promoted to the datum `heap of vertices and isotropy skew brace': this time, though, not all quiver skew braces are obtained from the datum of a heap and a skew brace. In \cref{sec:harder}, we provide a counterexample. 

In other words, it is not immediate to rewrite quiver skew brace theory in a way that overrides the quiver-theoretic aspect. This prompts the study of decompositions of quiver skew braces into geometrically simpler objects, and the notion of an \textit{unprunable} quiver skew brace---that is, simple with respect to ideal bundles.
\subsection*{Notations and conventions} All notations will be compliant with those of \cite{ferri2025splitandFIT}. We denote quivers $Q$ by $Q = \left( Q^1\twomapsright{\source}{\target} Q^0 \right)$, where $\source$ and $\target$ are the source and target maps respectively. Subscripts such as $\source_Q, \target_Q$ are added when due. Morphisms of quivers are pairs $f =  (f^1,f^0)\colon Q\to R$ with $f^1\colon Q^1\to R^1$ and $f^0\colon Q^0\to R^0$, such that $\source_Rf^1 = f^0\source_Q$ and $\target_R f^1 = f^0\target_Q$. Strong morphisms $f = (f^1, \id_\Lambda)$ over $\Lambda$ are defined as in \cite{ferri2025splitandFIT}. A quiver is \textit{Schurian} if every ordered pair of vertices has at most one arrow connecting them (thus a Schurian quiver $Q$ is equivalent to a relation on $Q^0$). We use the following notations:

\begin{tabular}{ll}
	$\Quiv$ & Category of quivers with all morphisms.\\
	$\Quiv_\Lambda$ & Category of quivers over $\Lambda$ with strong morphisms.\\
	$\Gpd$ & Category of groupoids with groupoid morphisms.\\
	$\Gpd_\Lambda$& Category of groupoids over $\Lambda$ with strong morphisms.\\
	$\GpdSchur$ & Full subcategory of $\Gpd$ consisting of Schurian groupoids.\\
	$\Gpdconn$ & Full subcategory of $\Gpd$ consisting of connected groupoids.
\end{tabular}

For the arrows in $Q$ from $A\subseteq Q^0$ to $B\subseteq Q^0$, we use the notation $Q(A,B)$ as in \cite{ferri2025splitandFIT}, omitting the curly brackets when $A= \{\lambda\}$ or $B=\{\mu\}$ is a singleton, and we denote the set of loops $Q(\lambda,\lambda)$ by $Q_\lambda$. We also introduce the notation $\St_Q(\lambda)= Q(\lambda, Q^0)$ for the \textit{outgoing star} at $\lambda\in Q^0$, following Brown \cite{BrownFibrations}.

The category $\Quiv_\Lambda$ is monoidal, with monoidal product $\ot$ defined as follows (see \cite{matsumotoshimizu}):
\begin{gather*} (Q\ot R)^1 = \{ x\times y \in Q^1\times R^1 \mid \target(x)=\source(y) \},\quad (Q\ot R)^0 = \Lambda = Q^0 = R^0,\\
\source_{Q\ot R}(x\times y) = \source_Q(x),\quad \target_{Q\ot R}(x\times y)=\target_R(y).  \end{gather*}
We write $x\ot y$ for a pair $x\times y \in (Q\ot R)^1$ (we say that $x$ and $y$ are \textit{consecutive}).

As in \cite{ferri2025splitandFIT}, we use the \textit{Leibniz order} $fg= f\circ g$ when writing the composition of functions; but we use the \textit{anti-Leibniz} or \textit{diagrammatic order} for the composition of arrows in a groupoid: $xy$ is defined when $\target(x) = \source(y)$. This is consistent with the usual notations in topology for the concatenation of continuous paths.
\section{Groupoids and braided groupoids}\label{sec:QSKB}
\subsection{Groupoids} A \textit{groupoid} is a quiver $\Gg = \left( \Gg^1\twomapsright{\source}{\target}\Gg^0\right)$ endowed with a multiplication $m\colon \Gg\ot \Gg\to \Gg$, $m(x\ot y)= xy$, defined only for consecutive arrows, and with a family $\One_{\Gg^0} = \{1_\lambda\mid \lambda \in \Gg^0\}$ of loops, such that:
\begin{enumerate}
	\item $m(\id\ot m) = m(m\ot \id)$ (associativity);
	\item $1_\lambda$ is a loop on $\lambda$, satisfying $x1_\lambda = x$ and $1_\lambda y = y$ for all $x\in \Gg(\Gg^0,\lambda)$ and $y\in \Gg(\lambda,\Gg^0)$ (unit loops);
	\item for all $x\in \Gg^1$ there exists $x^{-1}\in \Gg^1$ such that $xx^{-1}= 1_{\source(x)}$ and $x^{-1}x = 1_{\target(x)}$ (inverses).
\end{enumerate}
A morphism of groupoids $f=(f^1, f^0 )\colon \Gg\to \Hh$ is a morphism of quivers satisfying $f^1(xy)= f^1(x)f^1(y)$ for all $x\ot y\in (\Gg\ot \Gg)^1$. A morphism of groupoids is strong (over $\Lambda=\Gg^0=\Hh^0$) if $f^0 = \id_\Lambda$.

\subsection{Heaps} Our approach in this paper pivots around the notion of \textit{heap}, o\-ri\-gi\-na\-ted from the work of Prüfer \cite{prufer1924theorie} (abelian case) and Baer \cite{baer1929einfuhrung}. Despite heaps containing roughly the same information as groups, the language of heaps can be more effective than the language of groups---or vice versa---for specific purposes. 

We follow the exposition of Brzezi\'nski \cite{BRZEZINSKITrussesParagons}.
\begin{definition}
	A \textit{heap} is a set $X$ equipped with a ternary operation $\points*{\blank,\blank,\blank}$ satisfying the \textit{associativity condition}
	\[  \points*{\points*{a,b,c},d,e}= \points*{a,b,\points*{c,d,e}} \]
	and the \textit{Mal'tsev conditions}
	\[  \points*{a,b,b}=a,\quad \points*{a,a,b}=b \]
	for all $a,b,c,d,e\in X$.
\end{definition}
It is often said that heaps are the `affine version' of groups. The following well-known fact explains this viewpoint.
\begin{proposition}
	Let $(G,+)$ be a group (not necessarily abelian), which we denote additively. Then $\points*{a,b,c}= a-b+c$ is a heap structure.
	
	Let $(G,\points*{\blank,\blank,\blank})$ be a heap. For all $b\in G$, then, $a +_b c:= \points*{a,b,c}$ is a group operation, with neutral element $b$.
	
	The two above connections are mutually inverse.
\end{proposition}
\subsection{Braided groupoids and quiver skew braces}\label{sec:BrGpdandQSKB} %This section summarises the interplay between braided groupoids and quiver skew braces. The entire content of this section can be found in  \cite{ferri2024dynamical,sheng2024postgroupoidsquivertheoreticalsolutionsyangbaxter}.

The following definition was given in explicit terms by Andruskiewitsch \cite{andruskiewitsch2005quiver}, but we reformulate it here compactly using \cite[Lemma 2.9 (a)]{andruskiewitsch2005quiver}. We borrow the term `semistrong action' from \cite{ferri2025splitandFIT}.
\begin{definition}[{\cite[Definition 2.8 and Lemma 2.9 (a)]{andruskiewitsch2005quiver}}]\label{braided_groupoid}
	A \emph{braided groupoid} is the datum of a groupoid $(\Gg,\cdot)$ and of two quiver morphisms $\rightharpoonup,\leftharpoonup\colon \Gg\ot \Gg\to \Gg$, a left and a right semistrong action respectively, satisfying 
	\begin{equation} \label{eq:braidedcomm} (a\rightharpoonup b) (a\leftharpoonup b)= ab  \end{equation}
	for all $a\ot b\in (\Gg\ot \Gg)^1$, and such that the map $a\ot b\mapsto (a\rightharpoonup b)\ot (a\leftharpoonup b)$ is a bijection $(\Gg\ot \Gg)^1\to (\Gg\ot \Gg)^1$.
\end{definition}
The morphism $r(a\ot b) = (a\rightharpoonup b) \ot (a\leftharpoonup b) $ is called a \textit{braiding} on $\Gg$. 	Because of \cite[Lemma 4.2]{ferri2025splitandFIT}, the actions $\rightharpoonup,\leftharpoonup$ can be handled with the same rules as the group actions, except that we have to make sure that every step is well-defined. With the same computations as in \cite[Theorem 2]{LYZ}, then, one can prove that two morphisms $\rightharpoonup,\leftharpoonup\colon \Gg\ot\Gg\to \Gg$ explicitly satisfy
	\begin{align}
		\label{bg1}& a\rightharpoonup bc = (a\rightharpoonup b) ((a\leftharpoonup b)\rightharpoonup c),\\
		\label{bg2}& ab\leftharpoonup c = (a \leftharpoonup (b \rightharpoonup c))(b \leftharpoonup c),\\
		&\label{bgun} a\rightharpoonup 1_{\target(a)} = 1_{\source(a)},\quad 1_{\source(a)}\leftharpoonup a = 1_{\target(a)}.
	\end{align}
	 for all $a\ot b\ot c$---in other words, $(\Gg,\Gg,\rightharpoonup,\leftharpoonup)$ is a matched pair of groupoids in the sense of \cite{andruskiewitsch2005quiver,mackenzie}; see \cite[Lemma 2.9 (a)]{andruskiewitsch2005quiver}.

The inverse of a braiding is also a braiding \cite[Lemma 2.10 (b)]{andruskiewitsch2005quiver}.

\begin{definition}\label{def:morphisms}
	Let $(\Gg, r)$ and $(\Hh, s)$ be braided groupoids. A (\textit{strong}) \textit{morphism} of braided groupoids $f\colon (\Gg,r)\to (\Hh,s)$ is a (strong) morphism of groupoids $f = (f^1 , f^0)$, satisfying the following condition: whenever $f^1(x)\ot f^1(y)$ is defined, $x\ot y$ is also defined, and
	\[  (f^1\times f^1)r^1(x\ot y) = s^1(f^1(x)\ot f^1(y)). \]
\end{definition}
\begin{remark}The definition of a morphism $f$ of braided groupoids implies, in particular, that $\im(f)$ is a subgroupoid of $\Hh$. Indeed, $xy$ is defined whenever $f^1(x)f^1(y)$ is defined, and one has $f^1(xy)= f^1(x)f^1(y)$, whence the fact that $\im(f)$ is closed under compositions. It is clearly closed under units and inverses.

It is important to recall that for a morphism of groupoids $f\colon \Gg\to \Hh$, in general $\im(f)$ might not be a subgroupoid of $\Hh$; see \cite{GroupsToGroupoidsBrown,TilsonCatAsAlg}.
\end{remark}
\begin{definition}\label{def:quasimorphisms}
	Let $(\Gg, r)$ and $(\Hh, s)$ be braided groupoids. A \textit{quasimorphism of braided groupoids} is a morphism of groupoids $f\colon \Gg\to \Hh$ such that
	\[  (f^1\times f^1)r^1(x\ot y) = s^1(f^1(x)\ot f^1(y)) \]
	whenever both sides are defined.
\end{definition}
\Cref{def:quasimorphisms} is weaker than \cref{def:morphisms} as it does not require that, if one side of the equation is defined, the other side is too. In principle, $f^1(x)\ot f^1(y)$ might be defined even though $x\ot y$ is not, and in such cases \cref{def:quasimorphisms} poses no condition on $s^1(f^1(x)\ot f^1(y))$.
\begin{remark}
	There exist quasimorphisms of braided groupoids that are not morphisms. We shall see such a specimen in \cref{ex:quasimorphism}, although some more theory needs to be built in preparation.
	
	Observe that a quasimorphism that is also strong must be a morphism: since $\target(f^1(x)) = \target(x)$ and $\source(f^1(y)) = \source(y)$, if $f^1(x)$ and $f^1(y)$ are consecutive then so are $x$ and $y$.
\end{remark}

The notion of \textit{quiver skew brace} represents the quiver-theoretic version (or `oidification') of Guarnieri and Vendramin's \textit{skew braces} \cite{guarnieri2017skew}. It was introduced by Sheng, Tang, and Zhu \cite{sheng2024postgroupoidsquivertheoreticalsolutionsyangbaxter} with the name \textit{skew bracoid}, which was however already in use (see Martin-Lyons and Truman \cite{martin2024skew}). There is evidence of a close connection between quiver skew braces and Martin-Lyons and Truman's skew bracoids: thus we try to avoid any confusion by provisionally adopting the name `quiver skew braces'.\footnote{We also recall that our paper \cite{ferri2024dynamical} contains a definition of \textit{quiver-theoretic skew braces}, slightly more general than the quiver skew braces defined here: that notion probably needs renaming too.}
\begin{definition}
	A \textit{quiver skew brace} $(\Gg, \{+_\lambda\}_{\lambda\in\Gg^0}, \cdot)$ is a groupoid $(\Gg,\cdot)$ together with, for every vertex $\lambda\in \Gg^0$, a (non necessarily abelian) group operation $+_\lambda$ on the star $\St_\Gg(\lambda)$, such that
	\begin{equation}\label{qtbc} a(b+_{\target(a)}c) = ab -_{\source(a)}a +_{\source(a)} ac  \end{equation}
	holds for all $a\ot b, a\ot c\in (\Gg\ot \Gg)^1$. 
	
	Equivalently\footnote{This equivalent formulation was pointed out to the author by P.\@ Saracco, and it appears to be well-known.}, it is a groupoid $\Gg$ with a heap structure $\points*{\blank,\blank,\blank}_\lambda$ (see \cite{BRZEZINSKITrussesParagons}) on each star $\St_\Gg(\lambda)$, satisfying
	\begin{equation}\label{eq:heapdistr} a\points*{b,c,d}_{\target(a)} = \points*{ab,ac,ad}_{\source(a)} \end{equation}
	for all $a\ot b, a\ot c, a\ot d\in (\Gg\ot \Gg)^1$. We call \eqref{eq:heapdistr} a \textit{left truss distributivity}, after \cite{BRZEZINSKITrussesParagons}.
\end{definition}
The above definition forces $1_\lambda$ to be the neutral element of $+_\lambda$; see \cite{ferri2024dynamical,sheng2024postgroupoidsquivertheoreticalsolutionsyangbaxter}. We call $(\St_\Gg(\lambda), +\lambda)$ the \textit{star-group} at $\lambda$. In the sequel, we shall sometimes write $+$ for $+_\lambda$, since the base vertex can usually be deduced from the context. A (\textit{strong}) \textit{morphism} of quiver skew braces is a (strong) morphism of groupoids that induces group homomorphisms between the additive star-groups. We denote by $\QSKB$, resp.\@ $\QSKB_\Lambda$, the category of quiver skew braces with morphisms, resp.\@ quiver skew braces over $\Lambda$ with strong morphisms.

The compatibility \eqref{qtbc} easily implies, as in \cite{vendramin2023skew}, the relations
\begin{equation}
	\label{eq:minus} a(-b+c)=a-ab+ac,\qquad  a(b-c)=ab-ac+a.
\end{equation}
\begin{proposition}[{\cite[Theorems 5.6, 5.8]{sheng2024postgroupoidsquivertheoreticalsolutionsyangbaxter}}] \label{prop:braided_to_qtsb}
	Given a groupoid $\Gg$, a braided groupoid structure on $\Gg$ is equivalent to a quiver skew brace structure on $\Gg$. Explicitly, given actions $\rightharpoonup,\leftharpoonup$ one defines $a+ b = a(a^{-1}\rightharpoonup b)$; and given a quiver skew brace structure $\{+_\lambda\}_{\lambda\in \Gg^0}$ one defines $a\rightharpoonup b = -a + ab$.
\end{proposition}

\begin{remark}
	Every set $\Lambda$ is associated with a \textit{coarse groupoid} $\widehat{\Lambda}$, having set of arrows $\Lambda\times \Lambda$, and the projections on the two components as source and target map respectively. As in \cite{ferri2025splitandFIT}, for $a,b\in\Lambda$ we denote the unique arrow $a\to b$ as $[a,b]$.
	
	The theory of braidings is easier to handle on coarse groupoids. It is known (see \cite{FerriShibukawa}) that braidings on $\widehat{\Lambda}$ correspond bijectively to heap structures on $\Lambda$. Namely, the heap structure $\points*{\blank,\blank,\blank}$ corresponds to the braiding
	\[ r([a,b]\ot [b,c]) = [a,\points*{a,b,c}]\ot [\points*{a,b,c}, c] , \]
	and to the quiver skew brace structure
	\[  [a,b]+_a [a,c] = [a,\points*{b,a,c}].  \]
\end{remark}

The following remark is crucial for \cref{sec:crossed}.
\begin{remark}\label{rem:fixavertex}
	Let $\Gg$ be a connected quiver skew brace. The left truss distributivity \eqref{eq:heapdistr} implies that every star-group $\St_\Gg(\lambda)$ can be retrieved from a fixed star-group $\St_\Gg(\zeta)$, once we fix a maximal coarse subgroupoid $\Hh$ of $\Gg$ (which is well known to always exist and be wide, see e.g.\@ \cite[Remark 5.6 and Lemma 5.7]{ferri2024dynamical}). If we denote by $[\zeta,\lambda]$ the unique arrow of $\Hh$ from $\zeta$ to $\lambda$, then clearly
	\[ a+_\lambda b = [\zeta,\lambda]^{-1}\points*{[\zeta,\lambda]a, [\zeta,\lambda], [\zeta,\lambda]b}_\zeta,  \]
	which is an expression of $+_\lambda$ in function of $+_\zeta$. This operation of describing the quiver skew brace structure in terms of a single star-group was called `parallelisation' in \cite{ferri2024dynamical}. It is actually functorial between quiver skew braces and zero-symmetric dynamical skew braces \cite[\S 5]{ferri2024dynamical}.
\end{remark}

Quiver skew braces and braided groupoids are essentially the same thing, thus it makes sense to define \textit{quasimorphisms} of quiver skew braces. These will not be needed in the rest of this paper, but before moving on we ought to prove that nontrivial instances of quasimorphisms exist. 

\begin{example}\hspace{-5pt}\footnote{This example was suggested by I.\@ Colazzo.}\label{ex:quasimorphism}
	Using quiver skew braces, we construct a quasimorphism of braided groupoids that is not a morphism. This counterexample is inspired by a counterexample to the First Isomorphism Theorem in $\Gpd$ given by Brown \cite{GroupsToGroupoidsBrown}. Let $\Gg$ be the coarse groupoid $\widehat{2}$, let $\Hh$ be the group $(\Z/4\Z, +)$, with $\Hh^0 = \{\bullet\}$, and define $f\colon \Gg\to \Hh$ as $f^0(0) = f^0(1) = \bullet$, $f^1([i,i])= 0$, $f^1([0,1])= 1$, $f^1([1,0]) = 3$:
	
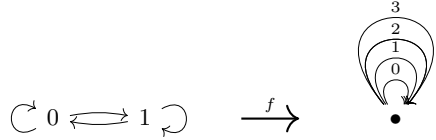
\[\begin{tikzcd}
		0 & 1 & {} & {} & \bullet
		\arrow[from=1-1, to=1-1, loop, in=150, out=210, distance=5mm]
		\arrow[bend left=10, from=1-1, to=1-2]
		\arrow[bend left=10, from=1-2, to=1-1]
		\arrow[from=1-2, to=1-2, loop, in=330, out=30, distance=5mm]
		\arrow[thick, "f", from=1-3, to=1-4]
		\arrow["1", from=1-5, to=1-5, loop, in=55, out=125, distance=10mm]
		\arrow["0", from=1-5, to=1-5, loop, in=60, out=120, distance=5mm]
		\arrow[from=1-5, to=1-5, loop, in=50, out=130, distance=15mm]
		\arrow["2", from=1-5, to=1-5, loop, in=50, out=130, distance=15mm]
		\arrow["3", from=1-5, to=1-5, loop, in=50, out=130, distance=20mm]
	\end{tikzcd}\]
	On $\Gg$, we consider the following braiding:
	\begin{gather*} [0,1]\ot [1,0]\mapsto [0,1]\ot [1,0],\quad [1,0]\ot [0,1]\mapsto [1,0]\ot [0,1],\\
	[a,b]\ot [b,b]\mapsto [a,a]\ot [a,b],\quad [a,a]\ot [a,b]\mapsto [a,b]\ot [b,b]\text{ for all  }a,b\in\{0,1\}. \end{gather*}
	This is given by the quiver skew brace structure 
\[  [a,b]+_a [a,c] =[a, b-a+c \mod 2].  \]
On $\Hh$, we consider the only skew brace structure with multiplicative group isomorphic to $(\Z/4\Z, +)$, and additive group isomorphic to $(\Z/2\Z\times \Z/2\Z, +)$ \cite[Proposition 2.4]{BachillerClassification}, where the isomorphism between the additive group of the skew brace and the group $\Z/2\Z\times \Z/2\Z$ is given by the map
\[ 0\mapsto (0,0),\quad 1\mapsto (0,1),\quad 2\mapsto (1,0),\quad 3\mapsto (1,1). \] 
The associated solution has table
\begin{center}
	\begin{tabular}{c|cccc}
		$s$ & 0   & 1   & 2   & 3   \\ \hline
		 0  &(0,0)&(1,0)&(2,0)&(3,0)\\
		 1  &(0,1)&(3,3)&(2,1)&(1,3)\\
		 2  &(0,2)&(1,2)&(2,2)&(3,2)\\
		 3  &(0,3)&(3,1)&(2,3)&(1,1)\\
	\end{tabular}
\end{center}
If $x\ot y$ exists in $\Gg$, one can check by hand that $s^1(f^1(x)\ot f^1(y)) = (f^1\times f^1) r^1(x\ot y)$. However, the fact that $f^1(x)\ot f^1(y)$ exists in $\Hh$ does not imply that $x\ot y$ exists in $\Gg$ (take for instance $x = y=[1,2]$). Therefore, $f$ is a quasimorphism but not a morphism of braided groupoids.
	\end{example}
\section{Ideals and quotients}\label{sec:ideals} In analogy to skew braces, we give notions of ideals and quotients for quiver skew braces. We give two notions, termed \textit{ideals} and \textit{ideal bundles}, which work well in $\QSKB$ and in $\QSKB_\Lambda$ respectively. 
\begin{definition}
	Let $(\Gg, \set{+_\lambda}_{\lambda\in\Lambda}, \cdot)$ be a quiver skew brace over $\Lambda$. An \textit{ideal} of $\Gg$ is a normal subgroupoid $\Ii$ of $\Gg$ such that $\St_\Ii(\lambda)$ is an additive subgroup of $\St_\Gg(\lambda)$ for all $\lambda\in \Ii^0$, and satisfying
	\begin{align}
		\label{eq:weak_ideal_comm}& a+_{\source(a)}\St_\Ii(\source(a)) = \St_\Ii(\source(a))+_{\source(a)} a\text{ (normality)},\\
		\label{eq:weak_ideal_invar}& a\rightharpoonup \St_\Ii(\target(a))\subseteq \St_\Ii(\source(a)) \text{ (invariance)},
	\end{align}
	for all $a\in \Gg^1$. An ideal $\Ii$ is an \textit{ideal bundle} if it is moreover a subgroup bundle.
\end{definition}
\begin{proposition}\label{prop:quotient_is_QSKB}
	Let $\Gg$ be a quiver skew brace, and $\Ii$ be an ideal of $\Gg$. Then, the quotient $\Gg/\Ii$ inherits a unique quiver skew brace structure such that the canonical projection $\pi\colon \Gg\to \Gg/\Ii$ is a morphism of quiver skew braces. 
	
	This is moreover a strong morphism over $\Gg^0$ if and only if $\Ii$ is an ideal bundle.
\end{proposition}
\begin{proof}
	Since $\Ii$ is normal in $\Gg$, the operation $\cdot$ induces a groupoid structure on $\Gg/\Ii$ (denoted again by $\cdot $) such that $\pi$ is a morphism of groupoids. Moreover, \eqref{eq:weak_ideal_comm} states that $(\St_\Ii(\lambda), +_\lambda)$ is a normal subgroup of $(\St_\Gg(\lambda), +_\lambda)$ for all $\lambda\in\Gg^0$. We need to check that two arrows $x,y\in\St_\Gg(\lambda)$ are equivalent in $\Gg/\Ii$ if and only if they are equivalent in the quotient group $\St_\Gg(\lambda)/\St_\Ii(\lambda)$. Let $x,y\in\St_\Gg(\lambda)$ be equivalent in $\Gg/\Ii$, thus $x = y  i$ for some $i\in \Gg(\target(y),\target(x))$. We search for an element $j\in\St_\Gg(\lambda)$ such that $x = y+j$, and then we prove that $j$ lies in $\St_\Ii(\lambda)$. By the definition of $\rightharpoonup$ one easily has $y+j = y  (y^{-1}\rightharpoonup j)$, thus the request that $x$ be equal to $ y+j$ is equivalent to the condition $ j =y\rightharpoonup (y^{-1}  x) =  y\rightharpoonup i$. Finally, $y\rightharpoonup i \in \St_\Ii(\lambda)$ holds by \eqref{eq:weak_ideal_invar}. This proves that the quotient is well-defined.
	
	It is immediate to notice that the compatibility \eqref{qtbc} passes to the quotient, hence $\Gg/\Ii$ is a quiver skew brace, and the canonical projection is clearly a morphism.
	
	If $\Ii$ is an ideal bundle, the quotient $\Gg/\Ii$ is a quiver over $\Gg^0$, and $\pi^0$ is the identity map on $\Gg^0$. Conversely if $\pi^0 = \id_{\Gg^0}$ then there is no arrow $x\in \Ii^1$ with $\source(x)\neq \target(x)$, otherwise $\source(x)$ and $\target(x)$ would be identified by $\pi^0$: therefore $\Ii$ is a bundle of loops, and hence an ideal bundle. 
\end{proof}
\begin{proposition}
	The ideals of $\Gg$ are exactly the kernels of morphisms $f\colon \Gg\to \Hh$ in $\QSKB$. The ideal bundles are exactly the kernels of morphisms in $\QSKB_{\Gg^0}$.
\end{proposition}
\begin{proof}
	Every ideal $\Ii$ is the kernel of the projection $\pi\colon \Gg\to \Gg/\Ii$. By \cref{prop:quotient_is_QSKB}, $\pi$ is a morphism in $\QSKB$, and in case $\Ii$ is an ideal bundle, the same projection is a morphism in $\QSKB_{\Gg^0}$.
	
	Conversely, let $f\colon \Gg\to\Hh$ be a morphism in $\QSKB$, and $\Ii$ be its kernel. For every $\lambda\in\Gg^0$, the morphism $f$ induces a homomorphism between the additive groups $\St_\Gg(\lambda)$ and $\St_\Hh(f^0(\lambda))$, whose kernel $\St_\Ii(\lambda) = \St_\Ii(\lambda)$ is a normal subgroup: thus $\Ii$ satisfies \eqref{eq:weak_ideal_comm}. As for \eqref{eq:weak_ideal_invar}, for $a\in \Gg^1$ and $k\in \St_\Ii(\target(a))$ one has
	\[ f^1(a\rightharpoonup k) = f^1(-a+ak) = -f^1(a) + f^1(a)1 =1, \]
	thus $a\rightharpoonup k \in \St_\Ii(\source(a))$. If moreover $f$ is a morphism in $\QSKB_{\Gg^0}$, then its kernel is a bundle of loops, and hence an ideal bundle.
\end{proof}

\begin{corollary}\label{lem:right_invariant}
	Let $\Ii$ be an ideal in $\Gg$. Then $\Ii$ is invariant under $\leftharpoonup$.
\end{corollary}
\begin{proof}
	Let $\Ii$ be the kernel of $f$. Let $i\in\Ii^1$ and $g\in\Gg^1$ such that $i\ot g$ is defined. Then
	\begin{align*}&f^1((i\rightharpoonup g)(i\leftharpoonup g)) = (f^1(i)\rightharpoonup f^1(g))f^1(i\leftharpoonup g)\\& = (1\rightharpoonup f^1(g))f^1(i\leftharpoonup g)= f^1(g)f^1(i\leftharpoonup g),\end{align*}
	while on the other hand
	\[ f^1((i\rightharpoonup g)(i\leftharpoonup g))= f^1(ig) =f^1(i)f^1(g), \]
	whence $f^1(i\leftharpoonup g) = 1$.
\end{proof}

\section{Products}\label{sec:products}

\subsection{Classical semidirect products \textit{à la} Brown} Given a group $G$, a set $\Lambda$, and an action $\triangleright$ of the groupoid $\widehat{\Lambda}$ on $G$, we recall e.g.\@ from Brown \cite{BrownFibrations,GroupsToGroupoidsBrown} that the set $G\times \widehat{\Lambda}^1$ can be given a groupoid structure, which is denoted by $G\rtimes \widehat{\Lambda}$ in \cite{ferri2025splitandFIT}, by setting
\begin{gather*} \source\Big(g\times [a,b]\Big) = a, \qquad \target\Big(g\times [a,b]\Big) = b,\\
	\Big(g\times [a,b]\Big) \Big(h\times [b,c]\Big) = g([a,b]\triangleright h)\times [a,c]. \end{gather*} 
	
	Using the same nomenclature as in \cite{ferri2025splitandFIT}, we also say `$\Gg$-module' meaning a quiver $Q$ with a semistrong action of the groupoid $\Gg$ on $Q$. We say that a small category $\Mm$ is a `$\Gg$-module algebra' if the quiver $\Mm$ is a $\Gg$-module with the action $\triangleright$, and moreover $g\triangleright (ab) = (g\triangleright a) (g\triangleright  b)$ for $g\in \Gg^1$, $a,b\in \Mm^1$.
	
	We can define a skew brace-by-heap semidirect product, giving rise to a quiver skew brace. We do not prove the following proposition immediately, since it will be a special case of \cref{prop:classicalsemidirect}.
\begin{proposition}
	Let $(G,+,\cdot)$ be a skew brace, $(\Lambda, \points*{\blank,\blank,\blank})$ be a heap, and $\triangleright$ be an action of the groupoid $\widehat{\Lambda}$ on $G$. On the groupoid $(G,\cdot)\rtimes \widehat{\Lambda}$ define the operation
	\[ (g\times [a,b])+(h\times[a,c]) = (g([a,b]\triangleright h))\times [a,\points*{b,a,c}].  \]
	This restricts to group operations for all stars $\St_{G\rtimes \widehat{\Lambda}}(a) = G\times \St_{\widehat{\Lambda}}(a)$. It makes $G\rtimes \widehat{\Lambda}$ into a quiver skew brace if and only if $(G,+)$ is a $\widehat{\Lambda}$-module algebra.
\end{proposition}

When $\triangleright$ is the trivial action, we call this a \textit{direct product}, and denote it by $G\times \widehat{\Lambda}$.

The skew brace-by-heap semidirect product can be slightly generalised, to a quiver skew brace-by-heap semidirect product. This will lead us to finding a quiver skew brace that has nontrivial ideals but only trivial ideal bundles (\cref{ex:unprunable_nonsimple}).

	Let $\Gg$ be a quiver skew brace, and $(\Lambda, \points*{\blank,\blank,\blank}$) be a heap. Let $\triangleright$ be an action of the groupoid $\widehat{\Lambda}$ on $\Gg$. Consider the semidirect product of groupoids $\Gg\rtimes_\triangleright \widehat{\Lambda}$ as in \cite[Definition 4.6]{ferri2025splitandFIT}, with underlying quiver
	\begin{gather*} (\Gg^1\times\widehat{\Lambda}^1 )\twomapsright{\Source}{\Target} (\Gg^0\times \Lambda)  \\ \Source(g\times [a,b]) = \source(g)\times a,\qquad \Target(g\times [a,b]) = ([b,a]\triangleright \target(g))\times b, \end{gather*}
	and with multiplication and units
	\[  (g\times [a,b])\cdot (h\times [b,c])  = g([a,b]\triangleright h) \times [a,c],\qquad 1_{\lambda, a}= 1_\lambda\times 1_a. \]

\begin{proposition}\label{prop:classicalsemidirect}
	Let $\Gg,\Lambda, \triangleright$ be as above. On the semidirect product $\Gg \rtimes_\triangleright\widehat{\Lambda}$, define an additive structure as
	\[  (g\times [a,b]) +_a (h\times [a,c]) = (g+_{\lambda} h)\times [a, \points*{b,a,c}]   \]
	for $\lambda\in\Gg^0$, $g,h \in \St_\Gg(\lambda)$. This is a quiver skew brace structure if and only if the bundle of additive star-groups $\bigsqcup_{\lambda\in \Gg^0} \St_\Gg(\lambda)$ is a $\widehat{\Lambda}$-module algebra.
\end{proposition}
\begin{proof}
	We check that $+$ is associative:
	\begin{align*}
		&\big((g\times [a,b])+(h\times[a,c])\big)+ (k\times[a,d])\\
		&= (g+h+k)\times [a, \points*{\points*{b,a,c},a,d}]\\
		&=(g+h+k)\times [a, \points*{b,a,\points*{c,a,d}}]\\
		&= (g\times [a,b])+\big((h\times[a,c])+ (k\times[a,d])\big).
	\end{align*}
	Clearly $1\times [a,a]$ is a neutral element for $\St_{\Gg\rtimes \widehat{\Lambda}}(a)$, and the additive inverse of $g\times [a,b]$ is $(-g)\times[a, \points*{a,b,a}]$. We check the compatibility between $\cdot$ and $+$: on one hand
	\begin{align*}&(g\times [a,b])\cdot \Big( h\times [b,c] + k\times [b,d] \Big)\\
		&=	g([a,b]\triangleright (h+k))\times [a, \points*{c,b,d}].
	\end{align*} 
	On the other hand,
	\begin{align*}
		&(g\times [a,b])(h\times [b,c]) - (g\times [a,b]) + (g\times [a,b])(k\times{b,d})\\
		&= \Big( g([a,b]\triangleright h)-g + g([a,b]\triangleright k)  \Big)\times [a, \points*{ \points*{c,a,\points*{a,b,a}},a,d}]\\
		&= \Big( g([a,b]\triangleright h)-g + g([a,b]\triangleright k)  \Big)\times [a, \points*{ \points*{\points*{c,a,a},b,a},a,d}]\\
		&=\Big( g([a,b]\triangleright h)-g + g([a,b]\triangleright k)  \Big)\times [a, \points*{ \points*{c,b,a},a,d}]\\
		&=\Big( g([a,b]\triangleright h)-g + g([a,b]\triangleright k)  \Big)\times [a, \points*{ c,b,d}]\\
		&=\big( g([a,b]\triangleright h+ [a,b]\triangleright k)  \big)\times [a, \points*{ c,b,d}].
	\end{align*}
	The first entries are equal for all $g\in \Gg(\Lambda, [a,b]\triangleright \source(h)) = \Gg(\Lambda, [a,b]\triangleright \source(k))$ if and only if  
	\[[a,b]\triangleright (h+k) = [a,b]\triangleright h+ [a,b]\triangleright k \]
	holds for all $\lambda\in\Gg^0$, $h,k\in \St_\Gg(\lambda)$, which is the module algebra condition on $\triangleright$.
\end{proof}

\begin{proposition}\label{prop:ideal_of_prod}
	Let $\Gg\rtimes_\triangleright \widehat{\Lambda}$ be a semidirect product of quiver skew braces. Then $\Gg\times \One_{\widehat{\Lambda}}= \bigsqcup_{a\in \Lambda} \Gg \times [a,a]$ is an ideal. In particular, it is a sub-quiver skew brace, isomorphic to $\bigsqcup_{a\in \Lambda} \Gg$.
\end{proposition}
\begin{proof}
	Since $\points*{a,a,a} = a$ for all $a$, it is clear that every star of $\Gg\rtimes \One_{\widehat{\Lambda}}$ is closed under the sum. Moreover, one has
	\[ (g\times [a,a]) \cdot (h\times [a,a]) = g([a,a]\triangleright h) \times [a,a]  = gh\times [a,a]  \]
	because every unit $[a,a] = 1_a$ in $\widehat{\Lambda}$ must act trivially (see e.g.\@ \cite[Lemma 4.2 (\textit{ii})]{ferri2025splitandFIT}), hence $\Gg \rtimes \One_{\widehat{\Lambda}}$ is a subgroupoid. Observe that, since every $1_a$ acts trivially, as a sub-quiver skew brace one has indeed $\Gg\rtimes \One_{\widehat{\Lambda}} = \Gg\times \One_{\widehat{\Lambda}} \cong \bigsqcup_{a\in\Lambda }\Gg$.
	
	We finally observe that $\Gg\times \One_{\widehat{\Lambda}}$ is stable under $\rightharpoonup$. One has
	\begin{align*}
		&- (g\times [a,b]) + (g\times [a,b])(h\times [b,b])\\
		&= (-g)\times [a,\points*{a,b,a}]  +  g([a,b]\triangleright h)\times [a,b]\\
		&= (-g+g([a,b]\triangleright h)) \times [a,\points*{\points*{a,b,a},a,b}]\\
		&= (-g+g([a,b]\triangleright h))\times [a,a],
		\end{align*}
		which lies in $\Gg\times \One_{\widehat{\Lambda}}$.	
\end{proof}

\subsection{Categorical semidirect products \textit{à la} Bourn and Janelidze}\label{sec:crossed} The categorical notion of semidirect product, in the sense of Bourn and Janelidze \cite{BournJanelidzeProtomodularityDescSemProd}, is a universal object corresponding to split epimorphisms. In the category $\SKB$ of skew braces, the semidirect product was defined by Facchini and Pompili \cite{FacchiniPompiliSemidirect}. In the category $\Gpd$ of groupoids, a categorical semidirect product was described in \cite{ferri2025splitandFIT}, where it was named `crossed product' to avoid confusion with the classical (and more geometrically motivated) notion of semidirect product defined by Brown \cite{BrownBookGroupoids}. A semidirect product in $\QSKB$ should truss the two respective semidirect products of $\SKB$ and of $\Gpd$. This is what we pursue here. 

\begin{definition}
	A digroup $(A,+,\cdot)$ is the datum of two group structures on the same set, with same unit $1$.
\end{definition}

Skew braces are in particular digroups.

\begin{definition}[{Facchini and Pompili \cite{FacchiniPompiliSemidirect}}] Let $(N,+, \cdot)$ and $(H,+,\cdot)$ be digroups. Suppose given:
	\begin{enumerate}
		\item a right action $\triangleleft$ of $(H,\cdot)$ on $N$, by group automorphisms of $(N,\cdot)$;
		\item a right action $\varphi$ of $(H,+)$ on $N$, by group automorphisms of $(N,+)$;
		\item a morphism of pointed sets $\gamma \colon (H, 1)\to (\Aut_{\Set^*}(N,1), \id)$, $h\mapsto \gamma_h$, where $\Set^*$ is the category of pointed sets.
	\end{enumerate}
The \textit{semidirect product} $H\ltimes N$ is defined as the set $H\times N$, with group operations
\begin{align*} (h,a)\cdot (k, b) & = (hk, a(b\triangleleft h)),\\
(h,a)+ ( k,b) & = (h + k,\gamma_{h+k}^{-1}( \varphi_{k}\gamma_{h}(a) + \gamma_{k}(b))).
  \end{align*}
\end{definition}

It can be proven that this is a digroup \cite[Theorem 4.1]{FacchiniPompiliSemidirect}, and it is the semidirect product in the category of digroups \cite[Theorem 3.1]{FacchiniPompiliSemidirect}. If $N$ and $H$ are skew braces, then $H\ltimes N$ is moreover a skew brace if and only if $\gamma\colon (N,\cdot)\to \Aut_{\Gp}(H, +)$ is a group homomorphism with image in the additive automorphisms of $H$, and the additional compatibility
\begin{equation}\label{eq:FPcomp}\begin{split}
	&\gamma_{k_1k_2 - k_1 + k_1k_3}^{-1}\!\Big(\! \varphi_{-k_1+k_1k_3}\big(\! \gamma_{k_1k_2}(\!(a_1\!\triangleleft\! k_2)a_2) - \gamma_{k_1}\!(a_1)  \!\big)+ \gamma_{k_1k_3}(\!(a_1\!\triangleleft\! k_3)a_3)  \!\Big)\\
	&= (a_1\triangleleft (k_2+k_3))\gamma_{k_2+k_3}^{-1}\Big( \varphi_{k_3}\gamma_{k_2}(a_2) + \gamma_{k_3}(a_3) \Big)
\end{split}\end{equation}
holds for all $a_i\in N$, $k_i\in H$ \cite[Proposition 4.2]{FacchiniPompiliSemidirect}. The condition \eqref{eq:FPcomp} arises naturally from requiring the skew brace compatibility. However, from our quiver-theoretic perspective, we will shed some new light on this rather technical condition (see e.g.\@ \cref{lem:FPcomp_heap_simplified} and \cref{lem:crossed_welldefined}).

Hereafter, we assume that both $N$ and $H$ are skew braces, and $\gamma$ is a homomorphism from $(N,\cdot)$ into the additive automorphisms of $H$.
\begin{remark}
	The inverses of $(h,a)$ with respect to the two group operations of $H\ltimes N$ are, respectively:
	\begin{align*}
		(h,a)^{-1}&= (h^{-1},a^{-1}\triangleleft h^{-1}), \\
		-(h,a)&= \big(-h,\, - \gamma_{-h}^{-1}\left( \varphi_{h}^{-1}\gamma_h(a) \right)\big).
	\end{align*}
	The former is simply the inverse of an element in a semidirect product of groups, while the latter is verified to be the additive inverse via an immediate computation.
\end{remark}

For the groupoidal setting, using a heap-theoretic formulation of these semidirect products will be inevitable, as we shall see in \cref{rem:inevitable}. Thus we first rewrite the additive side of Facchini and Pompili's semidirect product in the heap-theoretic form.
\begin{lemma}\label{lem:facchinipompili_heap}
	The heap structure induced by the group $(H\ltimes N, +)$ reads as follows:
	\begin{align*}&
		\points*{(h_1, a_1), ( h_2, a_2), ( h_3, a_3)} =\\
		& \Big( \points*{h_1, h_2, h_3},\; \gamma_{\points*{h_1, h_2, h_3}}^{-1}\left(  \points*{ \varphi_{h_3}\varphi_{h_2}^{-1}\gamma_{h_1}(a_1),\; \varphi_{h_3}\varphi_{h_2}^{-1}\gamma_{h_2}(a_2), \;\gamma_{h_3}(a_3) } \right) \Big),
	\end{align*}
	where we denote with the same brackets $\points*{h_1, h_2, h_3} = h_1 - h_2 + h_3$ also the heap structure of $(H,+)$.
\end{lemma}
\begin{proof}
	Using the definition of $(H\ltimes N, +)$ twice, it is immediate to compute
	\begin{align*}
		& (k_1,b_1) + (k_2,b_2) + (k_3,b_3) = \\
		& \Big( k_1+k_2+k_3,\; \gamma^{-1}_{k_1+k_2+k_3}\big( \varphi_{k_3}\varphi_{k_2}\gamma_{k_1}(b_1) + \varphi_{k_3}\gamma_{k_2}(b_2) + \gamma_{k_3}(b_3) \big) \Big).
	\end{align*} One concludes by substituting $(k_1, b_1) = (h_1, a_1)$,  $(k_2, b_2) = -(h_2, a_2)$,  $(k_3, b_3) = (h_3, a_3)$.
\end{proof}
Before moving on, let us rewrite the compatibility \eqref{eq:FPcomp} in a way that will be more useful in the groupoidal context. In order to lighten the notation, let us set
\[ \TT_{h_1, h_2, h_3}(b_1, b_2, b_3) = \gamma^{-1}_{\points*{h_1, h_2, h_3}}\points*{\varphi_{h_3}\varphi_{h_2}^{-1}\gamma_{h_1}(b_1), \varphi_{h_3}\varphi_{h_2}^{-1}\gamma_{h_2}(b_2), \gamma_{h_3}(b_3)}. \]
\begin{lemma}
	The compatibility \eqref{eq:FPcomp} can be rewritten as
	%\begin{equation}\label{eq:FPcomp_heap}\begin{split}
			%&\gamma^{-1}_{\points*{h_1, h_2, h_3}}\Big(\points*{ \varphi_{h_3}\varphi_{h_2}^{-1}\gamma_{h_1}(( \ell \triangleleft h_2^{-1}h_1)b_1) , \varphi_{h_3}\varphi_{h_2}^{-1}\gamma_{h_2}(\ell), \gamma_{h_3}((\ell\triangleleft h_2^{-1}\!h_3)b_3)}\Big)\\
			%&= \Big(\! \ell \triangleleft \points*{h_2^{-1}\!h_1, 1, h_2^{-1}h_3} \!\Big)\gamma^{-1}_{\points*{h_2^{-1}\!h_1, 1, h_2^{-1}\!h_3}} \Big(\!\points*{\varphi_{h_2^{-1}\!h_3}\gamma_{h_2^{-1}\!h_1}(b_1), 1, \gamma_{h_2^{-1}\!h_3}(b_3)}\!\Big).
	%\end{split}\end{equation}
	\begin{equation}\label{eq:FPcomp_heap}\begin{split}
		& \TT_{h_1, h_2, h_3}\Big((m\triangleleft h_2^{-1}h_1)b_1, m, (m\triangleleft h_2^{-1}h_3)b_3)  \Big)\\
		&= (m\triangleleft \points*{h_2^{-1}h_1, 1, h_2^{-1}h_3})\TT_{h_2^{-1}h_1, 1, h_2^{-1}h_3}(b_1, 1, b_3)
			\end{split}\end{equation}
for all $h_i \in H$, $m, b_1, b_3 \in N$.
\end{lemma}
\begin{proof}
	Immediate computations, by substituting $k_1 = h_2$, $k_2 = h_2^{-1}h_1$, $k_3 = h_2^{-1}h_3$, $a_1 = m$, $a_2 = b_1$, $a_3 = b_3$.
\end{proof}
The above relation can be further simplified.
\begin{lemma}\label{lem:TTiskinvariant}
	Under the assumption of \eqref{eq:FPcomp}, or equivalently \eqref{eq:FPcomp_heap}, one has $ \TT_{kh_1, kh_2, kh_3} = \TT_{h_1, h_2, h_3} $ for all $h_i, k\in H$.
\end{lemma}
\begin{proof} By definition of the heap structure in $H\ltimes N$, one has 
	\[  \TT_{h_1, h_2, h_3}(b_1, b_2, b_3) = \points*{h_1, h_2, h_3}^{-1}\points*{h_1b_1, h_2b_2, h_3b_3}. \]
	Since \eqref{eq:FPcomp} implies that $H\ltimes N$ is a skew brace---and hence the product distributes with respect to the heap operation---one has
	\begin{align*}
		&\TT_{kh_1, kh_2, kh_3}(b_1, b_2, b_3)\\
		&= \points*{kh_1, kh_2, kh_3}^{-1}\points*{kh_1b_1, kh_2b_2, kh_3b_3}\\
		&= \points*{h_1, h_2, h_3}^{-1}k^{-1}k\points*{h_1b_1, h_2b_2, h_3b_3}\\
		&= \TT_{h_1, h_2, h_3}(b_1, b_2, b_3).\qedhere
	\end{align*}
\end{proof}
\begin{lemma}\label{lem:FPcomp_heap_simplified}
	The relation \eqref{eq:FPcomp_heap}, and hence \eqref{eq:FPcomp}, is equivalent to 
	\begin{equation}\label{eq:FPcomp_heap_simplified}
		\begin{split}&\TT_{h_1, h_2, h_3}\Big((\ell\triangleleft h_1)b_1, \ell\triangleleft h_2, (\ell\triangleleft h_3) b_3\Big)\\ &= (\ell\triangleleft \points* {h_1, h_2, h_3}) \TT_{h_1, h_2, h_3}(b_1, 1, b_3). \end{split}
	\end{equation}
\end{lemma}
\begin{proof}
	It is clear that \eqref{eq:FPcomp_heap_simplified} implies \eqref{eq:FPcomp_heap}. Conversely, assuming \eqref{eq:FPcomp_heap}, \cref{lem:TTiskinvariant} holds true. In order to obtain \eqref{eq:FPcomp_heap_simplified} from \eqref{eq:FPcomp_heap}, then, it suffices to substitute $\ell = m\triangleleft h_2^{-1}$ and apply \cref{lem:TTiskinvariant}.
\end{proof}
\begin{remark}
	The relation
	\begin{equation}\label{eq:FPcomp_TRUE}
		\TT_{h_1, h_2, h_3} \Big( (\ell\triangleleft h_1)b_1, (\ell\triangleleft h_2)b_2, (\ell\triangleleft h_3)b_3 \Big) = (\ell\triangleleft \points*{h_1, h_2, h_3}) \TT_{h_1, h_2, h_3}(b_1, b_2, b_3)
	\end{equation}
	clearly implies \eqref{eq:FPcomp_heap_simplified}, and hence \eqref{eq:FPcomp_heap} and \eqref{eq:FPcomp}.
\end{remark}

We now briefly recall the description of the categorical semidirect product in $\Gpd$. Let $\Nn$ and $\Hh$ be groupoids, such that $\Hh^0\subseteq \Nn^0$ and every connected component of $\Nn$ contains exactly one vertex of $\Hh$. Let $\underline{\Nn} = \Nn^\circlearrowright(\Hh^0, \Hh^0)$, let $\triangleright$ be a left semistrong action of $\Hh$ on $\underline{\Nn}$ (see \cite[\S 4.1]{ferri2025splitandFIT}), let $\underline{a}\triangleleft h = h^{-1}\triangleright \underline{a}$. For clarity, we reserve the underline notation for the elements of $\underline{\Nn}$. With respect to these two actions $\triangleright$ and $\triangleleft$, suppose that $\underline{\Nn}$ satisfies the module bialgebra conditions.
\begin{definition}
	We define the \textit{balanced tensor product} $\Nn\otb \Hh\otb \Nn$ as the quiver
	\[ (\Nn\ot \Hh\ot \Nn ) / (\sim, =) \]
	where $\sim$ is generated by the relations
	\[ a\underline{b}\ot h\ot \underline{c}d \sim a\underline{b}(h\triangleright \underline{c})\ot h\ot d\sim a\ot h\ot (\underline{b}\triangleleft h)\underline{c}d .\]
\end{definition}
\begin{definition}
	The \textit{categorical semidirect product}, or \textit{crossed product} of $\Nn$ and $\Hh$ is the groupoid $\Nn\rcross{}{}\Hh\lcross{}{}\Nn$ having underlying quiver $\Nn\otb \Hh\otb \Nn$, multiplication
	\begin{align*} (a_1\otb h_1\otb b_1)(a_2\otb h_2\otb b_2) &= a_1 (h_1\triangleright \underline{b_1a_2})\otb h_1h_2\otb b_2\\
		&= a_1\otb h_1h_2 \otb (\underline{b_1a_2}\triangleleft h_2)b_2, \end{align*}
	and units
	\[ 1_\lambda=  a\otb 1_\mu\otb a^{-1} , \]
	where $\mu$ is the unique element $\mu\in \Hh^0$ that lies in the same connected component of $\Nn$ as $\lambda$, and $a$ is any arrow in $\Nn(\lambda,\mu)$. It can be proven that $1_\lambda$ does not depend on the choice of $a$.
\end{definition}

We now merge the previous notions in the following definition.

\begin{definition}\label{def:crossed}
	Let $\Nn$ and $\Hh$ be quiver skew braces (with the usual notations for the additive and multiplicative structures). Let $\Hh^0 \subseteq \Nn^0$, such that every connected component of $\Nn$ contains exactly one vertex of $\Hh$. Define $\underline{\Nn}$ as above. Suppose that there are:
	\begin{enumerate}
		\item a left semistrong action $\triangleright$ of $(\Hh,\cdot)$ on $(\underline{\Nn}, \cdot)$ by groupoid homomorphisms (so that, with the right action $\triangleleft$ defined as above, $\underline{\Nn}$ becomes a bimodule algebra);
		\item right actions $\varphi$ by group automorphisms of $(\St_\Hh(\lambda), +)$ on $(\St_\Nn(\lambda), +)$ for all $\lambda \in \Hh^0$;
		\item a semistrong action $\gamma$ of $(\Hh,\cdot)$ on $\underline{\Nn}$, such that $\gamma_h\colon \St_\Nn(\target(h))\to\St_\Nn(\source(h))$ is an additive homomorphism for all $h\in\Hh^1$. 
	\end{enumerate}
	Suppose moreover that the compatibility \eqref{eq:FPcomp_TRUE} holds for all $h_1, h_2, h_3\in\St_\Hh(\lambda)$, $\lambda\in\Hh^0$, $b_i\in \St_\Nn(\target(h_i))$, and $\ell \in \Nn_\lambda$.
	
	The \textit{categorical semidirect product} $\Nn\rcross{}{}\Hh\lcross{}{}\Nn$ is the groupoid $\Nn\rcross{}{}\Hh\lcross{}{}\Nn$ with the additive structure 
	\begin{align*}
		&\points*{a_1\otb h_1\otb b_1 , \;a_2\otb h_2\otb b_2,\; a_3\otb h_3\otb b_3}= \\
		&a_1 \otb \points*{h_1, h_2, h_3} \otb \gamma_{\points*{h_1, h_2, h_3}}^{-1}\Big(\Big.\\
		& \quad \Big.\points*{ \varphi_{h_3}\varphi_{h_2}^{-1}\gamma_{h_1}(b_1),\, \varphi_{h_3}\varphi_{h_2}^{-1}\gamma_{h_2}((\underline{a_1^{-1}a_2}\triangleleft h_2)b_2),\, \gamma_{h_3}((\underline{a_1^{-1}a_3}\triangleleft h_3)b_3)   }\Big)\\
		&=a_1\otb \points*{h_1, h_2, h_3}\otb \TT_{h_1, h_2, h_3} \Big(b_1,\; (\underline{a_1^{-1}a_2}\triangleleft h_2)b_2, \; (\underline{a_1^{-1}a_3}\triangleleft h_3)b_3 \Big),
	\end{align*}
	where the notation $\TT_{h_1, h_2, h_3}(x,y,z)$ is defined as before, and we observe that its definition makes as much sense in the quiver-theoretic context.
\end{definition}

\Cref{def:crossed} comes with many verifications to be done.

\begin{lemma}\label{lem:independentofa_1}
	The ternary operation of the semidirect product was written by choosing $a_1$ at the beginning, but we could have given an equivalent definition using $a_2$ or $a_3$. Indeed, one has
	\begin{align*}
		& a_1\otb \points*{h_1, h_2, h_3}\otb \TT_{h_1, h_2, h_3} (b_1, \; (\underline{a_1^{-1}a_2}\triangleleft h_2)b_2,\; (\underline{a_1^{-1}a_3}\triangleleft h_3)b_3)\\
		&= a_2\otb \points*{h_1, h_2, h_3}\otb \TT_{h_1, h_2, h_3} ((\underline{a_2^{-1}a_1}\triangleleft h_1)b_1, \; b_2,\; (\underline{a_2^{-1}a_3}\triangleleft h_3)b_3)\\
		&= a_3\otb \points*{h_1, h_2, h_3}\otb \TT_{h_1, h_2, h_3} ((\underline{a_3^{-1}a_1}\triangleleft h_1)b_1, \; (\underline{a_3^{-1}a_2}\triangleleft h_2)b_2,\; b_3).
	\end{align*}
\end{lemma}
\begin{proof}
	This result makes crucial use of \eqref{eq:FPcomp_TRUE}. One has
	\begin{align*}
		& a_1\otb \points*{h_1, h_2, h_3}\otb \TT_{h_1, h_2, h_3} (b_1, \; (\underline{a_1^{-1}a_2}\triangleleft h_2)b_2,\; (\underline{a_1^{-1}a_3}\triangleleft h_3)b_3)\\
		&= a_2 \underline{a_2^{-1}a_1}\otb \points*{h_1, h_2, h_3}\otb \TT_{h_1, h_2, h_3} (b_1, \; (\underline{a_1^{-1}a_2}\triangleleft h_2)b_2,\; (\underline{a_1^{-1}a_3}\triangleleft h_3)b_3)\\
		&= a_2 \otb \points*{h_1, h_2, h_3}\otb\\ &\hspace{2em} (\underline{a_2^{-1}a_1}\triangleleft \points*{h_1, h_2, h_3})\TT_{h_1, h_2, h_3} \big(b_1, \; (\underline{a_1^{-1}a_2}\triangleleft h_2)b_2,\; (\underline{a_1^{-1}a_3}\triangleleft h_3)b_3\big)\\
		\overset{\eqref{eq:FPcomp_TRUE}}&{=}a_2\otb \points*{h_1, h_2, h_3}\otb \TT_{h_1, h_2, h_3} ((\underline{a_2^{-1}a_1}\triangleleft h_1)b_1, \; b_2,\; (\underline{a_2^{-1}a_3}\triangleleft h_3)b_3),
	\end{align*}
	and similarly for $a_3$.
\end{proof}

%\begin{remark}\label{rem:FPcomp_revisited}The condition \eqref{eq:FPcomp_TRUE} assumed in \cref{def:crossed} can be rewritten, using the definition of the additive structure in $\Nn\rcross{}{}\Hh\lcross{}{}\Nn$, as\begin{equation}\label{eq:FPcomp_quiv_rewritten}\begin{split}		&\points*{(\ell\triangleleft h_2^{-1})\otb h_1\otb b_1,\; (\ell\triangleleft h_2^{-1})\otb h_2\otb 1,\; (\ell\triangleleft h_2^{-1})\otb h_3\otb b_3}\\	&= h_2\cdot \points*{ \ell\otb h_2^{-1}h_1\otb b_1,\; \ell\otb 1\otb 1,\; \ell\otb h_2^{-1}h_3\otb b_3 }.\end{split} 	\end{equation} This is easily obtained from \[1\otb \points*{h_1,h_2, h_3}\otb \text{LHS}\eqref{eq:FPcomp_heap} = 1\otb \points*{h_1,h_2, h_3}\otb \text{RHS}\eqref{eq:FPcomp_heap}.\]\end{remark}
\begin{lemma}\label{lem:crossed_welldefined}
	The expression 
	\begin{align*}
		&a_1 \otb \points*{h_1, h_2, h_3} \otb \gamma_{\points*{h_1, h_2, h_3}}^{-1}\Big(\Big.\\
		& \quad \Big.\points*{ \varphi_{h_3}\varphi_{h_2}^{-1}\gamma_{h_1}(b_1),\, \varphi_{h_3}\varphi_{h_2}^{-1}\gamma_{h_2}((\underline{a_1^{-1}a_2}\triangleleft h_2)b_2),\, \gamma_{h_3}((\underline{a_1^{-1}a_3}\triangleleft h_3)b_3)   }\Big)
	\end{align*}
	is well-defined, i.e.\@ it does not depend on the chosen representatives $a_i\ot h_i\ot b_i$ of the equivalence classes $a_i\otb h_i\otb b_i$.
\end{lemma}
\begin{proof} 
	We first check that, for every loop $\underline{\ell}$ on $\source(h_i)=\lambda$, the ternary operation applied to \[a_1\underline{\ell}\otb h_1\otb b_1,\; a_2\otb h_2\otb b_2,\; a_3\otb h_3\otb b_3,\] yields the same result as if applied to \[a_1\otb h_1\otb (\underline{\ell}\triangleleft h_1)b_1,\; a_2\otb h_2\otb b_2,\; a_3\otb h_3\otb b_3.\]
	One has
	\begin{align*}
		&\points*{a_1\underline{\ell}\otb h_1\otb b_1, a_2\otb h_2\otb b_2, a_3\otb h_3\otb b_3 }\\
		&=a_1\underline{\ell}\otb \points*{h_1, h_2, h_3}\otb \TT_{h_1, h_2, h_3} \Big(b_1, (\underline{\ell^{-1}a_1^{-1}a_2^{-1}}\triangleleft h_2)b_2,  (\underline{\ell^{-1}a_1^{-1}a_3^{-1}}\triangleleft h_3)b_3 \Big)\\
		& = a_1\otb \points*{h_1, h_2, h_3}\otb \\&\hspace{2em} (\underline{\ell}\triangleleft \points*{h_1, h_2, h_3}) \TT_{h_1, h_2, h_3}\Big(b_1, (\underline{\ell^{-1}a_1^{-1}a_2^{-1}}\triangleleft h_2)b_2,  (\underline{\ell^{-1}a_1^{-1}a_3^{-1}}\triangleleft h_3)b_3 \Big)\\
		&= a_1\otb \points*{h_1, h_2, h_3}\otb\\&\hspace{2em} \TT_{h_1, h_2, h_3} \Big( (\underline{\ell}\triangleleft h_1)b_1, (\underline{\ell}\triangleleft h_2)(\underline{\ell^{-1}a_1^{-1}a_2}\triangleleft h_2)b_2, (\underline{\ell}\triangleleft h_3) (\underline{\ell^{-1}a_1^{-1}a_3}\triangleleft h_3)b_3 \Big)\\
		&= a_1\otb \points*{h_1, h_2, h_3}\otb \TT_{h_1, h_2, h_3} \Big( (\underline{\ell}\triangleleft h_1) b_1, (\underline{a_1^{-1}a_2}\triangleleft h_2)b_2, (\underline{a_1^{-1}a_3}\triangleleft h_3)b_3\Big) \\
		&=\points*{a_1\otb h_1\otb (\underline{\ell}\triangleleft h_1)b_1, a_2\otb h_2\otb b_2, a_3\otb h_3\otb b_3}.
	\end{align*}
	We can use \cref{lem:independentofa_1} to rewrite the ternary operation in terms of $a_2$ or $a_3$: the same argument as above, then, shows that the ternary operation is well-defined under multiplying a loop to $a_2$ or $a_3$.
	
	We should now check that the ternary operation is well-defined under multiplying $b_i$ by a loop on the left. But every such loop $\underline{m}$ would be of the form $\underline{\ell}\triangleleft h_i$ for $\underline{\ell} = \underline{m}\triangleleft h_i^{-1}$, thus this verification can be made fall under the previously treated cases.
\end{proof}

\begin{proposition}\label{prop:crossed_is_qskb}
	\Cref{def:crossed} defines indeed a quiver skew brace.
\end{proposition}
\begin{proof}
	We begin by checking that the ternary operation---which we now know being well-defined---is a heap structure. The Mal'tsev axiom
	\begin{align*}
		&\points*{a\otb h\otb b, a\otb h\otb b, a'\otb h'\otb b'}\\
		&= a\otb \points*{h,h,h'}\otb \TT_{h,h,h'}(b,( \underline{aa^{-1}}\triangleleft h)b, (\underline{a^{-1}a'}\triangleleft h')b')\\
		&= a\otb h'\otb \gamma^{-1}_{\points*{h,h,h'}}\Big( \points*{\gamma_h(b),\gamma_h(b),\gamma_{h'}((\underline{a^{-1}a'}\triangleleft h')b')}\Big)\\
		&= a\otb h'\otb \gamma^{-1}_{h'} \gamma_{h'}((\underline{a^{-1}a'}\triangleleft h')b')\\
		&= a\otb h' \otb ((\underline{a^{-1}a'}\triangleleft h')b')\\
		&= a\underline{a^{-1}a'}\otb h'\otb b'\\
		&= a'\otb h'\otb b'
	\end{align*} is readily verified, and same for the other Mal'tsev axiom. As for the associativity, one has (writing the ternary operation with respect to $a_3$ as in \cref{lem:independentofa_1})
	\begin{align*}
		&\points*{a_1\otb h_1\otb b_1, a_2\otb h_2\otb b_2, \points*{a_3\otb h_3\otb b_3, a_4\otb h_4\otb b_4, a_5\otb h_5\otb b_5} }\\
		&= a_3\otb \points*{h_1, h_2, \points*{h_3, h_4, h_5}}\otb \TT_{h_1, h_2, \points*{h_3, h_4, h_5}}\Big( (a_3^{-1}a_1\triangleleft)b_1, (a_3^{-1}a_2\triangleleft h_2)b_2, \Big.\\ & \hspace{2em}\Big.\TT_{h_3, h_4, h_5}\Big(  b_3, (a_3^{-1}a_4\triangleleft h_4)b_4, (a_3^{-1}a_5\triangleleft h_5)b_5\Big) \Big)\\
		&= a_3\otb \points*{h_1, h_2, \points*{h_3, h_4, h_5}}\otb \gamma^{-1}_{\points*{ h_1, h_2, \points*{h_3, h_4, h_5}}}\Big(\Big.\\
		& \hspace{2em} \Big\langle \varphi_{\points*{h_3, h_4, h_5}}\varphi^{-1}_{h_2} \gamma_{h_1} ((a_3^{-1}a_1\triangleleft h_1)b_1),\;  \varphi_{\points*{h_3, h_4, h_5}}\varphi^{-1}_{h_2} \gamma_{h_2} ((a_3^{-1}a_2\triangleleft h_2)b_2), \Big.\\
		& \hspace{2em} \gamma_{\points*{h_3,h_4,h_5}}\gamma^{-1}_{\points*{h_3,h_4,h_5}}\Big( \big\langle \varphi_{h_5}\varphi^{-1}_{h_4}\gamma_{h_3}(b_3), \big.\Big.\Big.\\ &\hspace{2em}\Big.\Big.\big. \varphi_{h_5}\varphi^{-1}_{h_4}\gamma_{h_4}((a_3^{-1}a_4\triangleleft h_4)b_4),\; \varphi_{h_5}\varphi^{-1}_{h_4}\gamma_{h_5}((a_3^{-1}a_5\triangleleft h_5)b_5)  \big\rangle\Big) \Big\rangle\Big)\\
		%%%%%%%%%%%
		&= a_3\otb \points*{\points*{h_1, h_2, h_3}, h_4, h_5}\otb \gamma^{-1}_{\points*{\points*{h_1, h_2, h_3}, h_4, h_5}}\Big(\Big.\\
		& \hspace{2em} \Big\langle \big\langle  \varphi_{\points*{h_3, h_4, h_5}}\varphi^{-1}_{h_2} \gamma_{h_1} ((a_3^{-1}a_1\triangleleft h_1)b_1),\;  \varphi_{\points*{h_3, h_4, h_5}}\varphi^{-1}_{h_2} \gamma_{h_2} ((a_3^{-1}a_2\triangleleft h_2)b_2), \big.\Big.\\
		& \hspace{2em}\big. \varphi_{h_5}\varphi^{-1}_{h_4}\gamma_{h_3}(b_3) \big\rangle\Big.\Big. ,\\ &\hspace{2em}\Big.\Big. \varphi_{h_5}\varphi^{-1}_{h_4}\gamma_{h_4}((a_3^{-1}a_4\triangleleft h_4)b_4),\; \varphi_{h_5}\varphi^{-1}_{h_4}\gamma_{h_5}((a_3^{-1}a_5\triangleleft h_5)b_5)  \Big\rangle\Big)\\
		%%%%%%%%%%%
		&= a_3\otb \points*{\points*{h_1, h_2, h_3}, h_4, h_5}\otb \gamma^{-1}_{\points*{\points*{h_1, h_2, h_3}, h_4, h_5}}\Big(\Big.\\
		& \hspace{2em} \Big\langle \big\langle  \varphi_{h_5}\varphi^{-1}_{h_4}\varphi_{h_3}\varphi^{-1}_{h_2} \gamma_{h_1} ((a_3^{-1}a_1\triangleleft h_1)b_1),\;  \varphi_{h_5}\varphi^{-1}_{h_4}\varphi_{h_3}\varphi^{-1}_{h_2} \gamma_{h_2} ((a_3^{-1}a_2\triangleleft h_2)b_2), \big.\Big.\\
		& \hspace{2em}\big. \varphi_{h_5}\varphi^{-1}_{h_4}\gamma_{h_3}(b_3) \big\rangle\Big.\Big. ,\\ &\hspace{2em}\Big.\Big. \varphi_{h_5}\varphi^{-1}_{h_4}\gamma_{h_4}((a_3^{-1}a_4\triangleleft h_4)b_4),\; \varphi_{h_5}\varphi^{-1}_{h_4}\gamma_{h_5}((a_3^{-1}a_5\triangleleft h_5)b_5)  \Big\rangle\Big)\\
		%%%%%%%%%%%%%%%%%
		%%%%%%%%%%%%%%%%%
		%%%%%%%%%%%%%%%%%
		&= a_3\otb \points*{\points*{h_1, h_2, h_3}, h_4, h_5}\otb \gamma^{-1}_{\points*{\points*{h_1, h_2, h_3}, h_4, h_5}}\Big(\Big.\\
		& \hspace{2em} \Big\langle  \varphi_{h_5}\varphi^{-1}_{h_4}\gamma_{\points*{h_1, h_2, h_3}}\gamma^{-1}_{\points*{h_1, h_2, h_3}} \big(\big.\Big.\\
		&\hspace{2em}\big.\big\langle
		\varphi_{h_3}\varphi^{-1}_{h_2} \gamma_{h_1} ((a_3^{-1}a_1\triangleleft h_1)b_1), \;   \varphi_{h_3}\varphi^{-1}_{h_2} \gamma_{h_2} ((a_3^{-1}a_2\triangleleft h_2)b_2), \; \gamma_{h_3}(b_3) \big\rangle \big), \\
	 &\hspace{2em}\Big.\Big. \varphi_{h_5}\varphi^{-1}_{h_4}\gamma_{h_4}((a_3^{-1}a_4\triangleleft h_4)b_4),\; \varphi_{h_5}\varphi^{-1}_{h_4}\gamma_{h_5}((a_3^{-1}a_5\triangleleft h_5)b_5)  \Big\rangle\Big)\\
	 &= \points*{\points*{a_1\otb h_1\otb b_1, a_2\otb h_2\otb b_2, a_3\otb h_3\otb b_3}, a_4\otb h_4\otb b_4, a_5\otb h_5\otb b_5}.
	\end{align*}
	Finally, we prove that the groupoid multiplication distributes with respect to this heap structure (we assume the composition to be well-defined):
	\begin{align*}
		&(a\otb h\otb b)\cdot \points*{ a_1\otb h_1\otb b_1,\; a_2\otb h_2\otb b_2,\;  a_3\otb h_3\otb b_3  }\\
		&= (a\otb h\otb b)\cdot \Big( a_1\otb \points*{h_1, h_2, h_3}\otb \TT_{h_1, h_2, h_3}(b_1, b_2, b_3) \Big)\\
		&= a\otb h\points*{h_1, h_2, h_3}\otb \big( ba_1\triangleleft  \points*{h_1, h_2, h_3}\big) \TT_{h_1, h_2, h_3}(b_1, b_2, b_3)\\
		\overset{\eqref{eq:FPcomp_TRUE}}&{=} a\otb \points*{hh_1, hh_2, hh_3}\otb \TT_{h_1, h_2, h_3}\big( (ba_1\triangleleft h_1)b_1, (ba_1\triangleleft h_2)b_2, (ba_1\triangleleft h_3)b_3 \big)\\
		&= \points*{ a\otb hh_1\otb (ba_1\triangleleft h_1)b_1,\;  a\otb hh_2\otb (ba_1\triangleleft h_2)b_2, a\otb hh_3\otb (ba_1\triangleleft h_3)b_3}\\
		&=  \points*{ (a\otb h\otb b)(a_1\otb h_1\otb b_1),\; (a\otb h\otb b)(a_2\otb h_2\otb b_2),\;  (a\otb h\otb b)(a_3\otb h_3\otb b_3)  }.
	\end{align*}
\end{proof}

We finally claim that this is the categorical semidirect product in $\QSKB$.
\begin{theorem}\label{thm:is_semidirect}
	For every split short exact sequence of quiver skew braces
	\[\begin{tikzcd}
		{\One_{\Gg^0}} & \Nn & \Gg & \Hh & 1.
		\arrow[from=1-1, to=1-2]
		\arrow[from=1-2, to=1-3]
		\arrow[from=1-3, to=1-4, "f", shift left=1]
		\arrow[from=1-4, to=1-3, "s", shift left=1]
		\arrow[from=1-4, to=1-5]
	\end{tikzcd}\]
	one has $\Gg \cong \Nn\rcross{}{}\Hh\lcross{}{}\Nn$ as quiver skew braces.
	
	The actions $\triangleright,\triangleleft$ are given by the conjugation of loops in $\Gg$. The right actions $\varphi$ are the additive right-conjugations in the additive star-groups: $\varphi_h(a)= -h+a+h$. Finally, the map $\gamma$ is the restriction of the action $\rightharpoonup$ in $\Gg$, namely $\gamma_h(a) = -h + ha$ for all $h\in \Hh^1$, $a\in \St_{\Nn}(\target(h))$.
\end{theorem}
\begin{proof}
	For simplicity, let us identify $\Hh$ with its isomorphic image via $s$, which is a sub-quiver skew brace of $\Gg$. First of all, from \cite[Theorem 5.4]{ferri2025splitandFIT} we know that $\Gg\cong \Nn\rcross{}{}\Hh\lcross{}{}\Nn$ as groupoids, precisely with the actions $\triangleright,\triangleleft$ given above. Moreover, since the kernel $\Nn$ of $f$ is an ideal, we know that $(\St_\Nn(\lambda), +)$ is normal in $(\St_\Gg(\lambda), +)$, thus in particular the additive right-conjugation by $h$ restricts to a group homomorphism $\varphi_h\colon \St_\Nn(\lambda)\to \St_\Nn(\lambda)$ for all $\lambda\in\Hh^0$. This is the desired additive right action.
	
	As in \cite[\S 5.2]{ferri2025splitandFIT}, every element of $\Gg$ can be written as a product $ahb$, $a,b\in\Nn^1$, $h\in \Hh^1$. The triple $a\ot h\ot b$ is not unique, but the class $a\otb h\otb b$ is.
	
	We now tackle the hardest part, namely we need to describe the additive structure of $\Gg$ in terms of the additive structures of $\Nn$ and $\Hh$, and prove that this yields exactly the additive structure of the categorical semidirect product.
	
	We begin by writing $a_1 h_1 b_1 - a_2 h_2 b_2 + a_3h_3b_3$ as
	\begin{align}\nonumber
		a_1 h_1 b_1 - a_2 h_2 b_2+ a_3h_3b_3&= a_1 h_1 b_1 - a_1\underline{a_1^{-1}a_2} h_2 b_2+ a_1\underline{a_1^{-1}a_3}h_3b_3 \\ \nonumber
		&= a_1(  h_1 b_1 - \underline{a_1^{-1}a_2} h_2 b_2+ \underline{a_1^{-1}a_3}h_3b_3 )\\ \label{eq:bring_a_1_out}
		&= a_1 \big(  h_1 b_1 - h_2 (\underline{a_1^{-1}a_2}\triangleleft h_2) b_2+ h_3(\underline{a_1^{-1}a_3}\triangleleft h_3)b_3 \big),
	\end{align}
	where the last step follows from the fact that every connected component of $\Nn$ contains exactly one vertex of $\Hh$, hence $\target(a_1) = \target(a_2) = \target(a_3)\in \Hh^0$, thus $\underline{a_1^{-1}a_2}$ and $\underline{a_1^{-1}a_3}$ are loops, on which $h_2$ and $h_3$ can act. Of course, instead of bringing $a_1$ out of the parentheses, we could have chosen to do so with $a_2$ or $a_3$ (\textit{a posteriori}, this will justify \cref{lem:independentofa_1}).
	
	Observe that all three terms of $h_1 b_1 - h_2 (\underline{a_1^{-1}a_2}\triangleleft h_2) b_2+ h_3(\underline{a_1^{-1}a_3}\triangleleft h_3)b_3$ are the product of an arrow in $\Hh$ with an arrow in $\Nn$. 
	
	The next step, then, is being able to write an element of the form 
	\[ \points*{h_1a_1, h_2a_2, h_3a_3}=h_1a_1 - h_2a_2 + h_3 a_3 \]
	as a product $ha$, for $h, h_i \in \Hh^1$, $a, a_i \in \Nn^1$. One trivially has
	\[ \points*{h_1a_1, h_2a_2, h_3a_3} = \points*{h_1, h_2, h_3} \points*{h_1, h_2, h_3}^{-1} \points*{h_1a_1, h_2a_2, h_3a_3},\]
	where $\points*{h_1, h_2, h_3}$ lies in $\Hh$. It is easy to observe that $\points*{h_1, h_2, h_3}^{-1} \points*{h_1a_1, h_2a_2, h_3a_3}$ lies in $\Nn = \ker(f)$, since $a_1, a_2, a_3$ lie in $\ker(f)$, and hence
	\begin{align*}&f^1 (\points*{h_1, h_2, h_3}^{-1} \points*{h_1a_1, h_2a_2, h_3a_3})\\&= \points*{f^1(h_1), f^1(h_2), f^1(h_3)}^{-1}\points*{f^1(h_1a_1), f^1(h_2a_2), f^1(h_3a_3)}\\
		&= \points*{f^1(h_1), f^1(h_2), f^1(h_3)}^{-1}\points*{f^1(h_1)f^1(a_1), f^1(h_2)f^1(a_2), f^1(h_3)f^1(a_3)}\\
	&= \points*{f^1(h_1), f^1(h_2), f^1(h_3)}^{-1}\points*{f^1(h_1), f^1(h_2), f^1(h_3)}\\
	&= 1.\end{align*}
	Using $\gamma_h(a) = -h+ha$ and $\varphi_h(a) = -h+a+h$, we compute
	\begin{align*}
		&\points*{h_1, h_2, h_3} \gamma^{-1}_{\points*{h_1, h_2, h_3}} \Big( \varphi_{h_3}(\varphi{h_2})^{-1}\gamma_{h_1}(a_1) - \varphi_{h_3}(\varphi_{h_2})^{-1}\gamma_{h_2}(a_2) + \gamma_{h_3}(a_3) \Big)\\
		&= \points*{h_1, h_2, h_3}\gamma^{-1}_{\points*{h_1, h_2, h_3}} \Big( -h_3 + h_2 - h_1 + h_1a_1-h_2 +h_3 -h_3 \Big.\\ & \hspace{12em}\Big.+h_2 -h_2a_2 +h_2 - h_2 +h_3 - h_3 + h_3a_3 \Big)\\
		&= \points*{h_1, h_2, h_3}\gamma^{-1}_{\points*{h_1, h_2, h_3}} \Big( -\points*{h_1, h_2, h_3} + \points*{h_1a_1, h_2a_2, h_3a_3}\Big)\\
		&= \points*{h_1, h_2, h_3} \Big( -\points*{h_1, h_2, h_3}^{-1} + \points*{h_1, h_2, h_3}^{-1} \big(-\points*{h_1, h_2, h_3} \big.\Big.\\ &\hspace{20em}\Big.\big. +\points*{h_1a_1, h_2a_2, h_3a_3}  \big) \Big)\\
		\overset{\eqref{eq:minus}}&{=}\points*{h_1, h_2, h_3} \Big( -\points*{h_1, h_2, h_3}^{-1}  + \points*{h_1, h_2, h_3}^{-1} \Big.\\ & \hspace{7em}\Big.- \points*{h_1, h_2, h_3}^{-1}\points*{h_1, h_2, h_3} +\points*{h_1, h_2, h_3}^{-1}\points*{h_1a_1, h_2a_2, h_3a_3}  \Big)\\
		&= \points*{h_1, h_2, h_3} \points*{h_1, h_2, h_3}^{-1}\points*{h_1a_1, h_2a_2, h_3a_3} \\
		&= \points*{h_1a_1, h_2a_2, h_3a_3}. 
	\end{align*}
	Therefore, the heap operation $\points*{h_1a_1, h_2a_2, h_3a_3}$ is expressed exactly as in \cref{def:crossed}. We finally need to express $\points*{a_1h_1b_1, a_2h_2b_2, a_3h_3b_3}$ as in \cref{def:crossed}. From the above, we have
	\begin{align*}
		&\points*{a_1h_1b_1, a_2h_2b_2, a_3h_3b_3}\\
		\overset{\eqref{eq:bring_a_1_out}}&{=} a_1 \points*{  h_1 b_1,\; h_2 (\underline{a_1^{-1}a_2}\triangleleft h_2) b_2,\; h_3(\underline{a_1^{-1}a_3}\triangleleft h_3)b_3 }\\
		&= a_1  \points*{h_1, h_2, h_3} \gamma_{\points*{h_1, h_2, h_3}}^{-1}\Big(\Big.\\
		& \quad \Big.\points*{ \varphi_{h_3}(\varphi_{h_2})^{-1}\gamma_{h_1}(b_1),\, \varphi_{h_3}(\varphi_{h_2})^{-1}\gamma_{h_2}((\underline{a_1^{-1}a_2}\triangleleft h_2)b_2),\, \gamma_{h_3}((\underline{a_1^{-1}a_3}\triangleleft h_3)b_3)   }\Big)
	\end{align*}
	which is the same expression appearing in \cref{def:crossed}, where $a_1$ lies in $\Nn$, $\points*{h_1, h_2, h_3}$ lies in $\Hh$, and 
	\begin{align*}&\gamma_{\points*{h_1, h_2, h_3}}^{-1}\Big( \Big.\Big\langle \varphi_{h_3}(\varphi_{h_2})^{-1}\gamma_{h_1}(b_1),\Big.\Big.\\ &\Big.\Big.\hspace{8em} \varphi_{h_3}(\varphi_{h_2})^{-1}\gamma_{h_2}((\underline{a_1^{-1}a_2}\triangleleft h_2)b_2),\; \gamma_{h_3}((\underline{a_1^{-1}a_3}\triangleleft h_3)b_3)   \Big\rangle\Big)\end{align*}
	lies in $\Nn$, as desired.
\end{proof}

\begin{remark} Suppose given a split short exact sequence as in \cref{thm:is_semidirect}. Observe that $f^1$ and $s^1$ restrict to group homomorphisms 
	\[ \begin{tikzcd}
		\St_\Gg(a)\ar[r,"{f^1}"] & \St_\Hh(f^1(a))\ar[r, "{s^1}"]&
		\St_\Gg(s^0f^0 (a)),
	\end{tikzcd}  \] 
	but $a\neq s^0 f^0 (a)$ in general---and hence, they do not necessarily induce split epimorphisms of groups. More precisely,  $s^0 f^0 (a)$ is the unique vertex of $\Hh^0$ that lies in the same connected component of $\Nn$ as $a$.\end{remark}

\begin{remark}\label{rem:inevitable}
	The entire reason why we use a heap-theoretic viewpoint in this paper, lies in \eqref{eq:bring_a_1_out}. In general, we would not know how to rewrite a simple sum $a h_1b_1 + a h_2b_2$ as a product $ahb$. We can only do so if $a$ is a loop on $\source(a) \in \Hh^0$. Indeed, the product $\cdot$ does \textit{not} distribute with respect to sums, but with respect to the heap operation. Thus, a group-theoretic formulation of the additive structure for $\Nn\rcross{}{}\Hh\lcross{}{}\Nn$ seems prohibitive. On the contrary, since we know how to rewrite $ah_1b_1 - ah_2 b_2 + ah_3 b_3$ in the form $ahb$, a heap-theoretic formulation of the additive structure for $\Nn\rcross{}{}\Hh\lcross{}{}\Nn$ comes as natural.
	
	The bottom line is: we would not be using heaps in this paper, if it were not, to the best of our knowledge, indispensable. 
	
	We do not exclude that a purely group-theoretic formulation may exist, and we would be greatly interested in knowing about it.   
\end{remark}

In the special case when $\lambda\in\Hh^0$, we can give a clean group-theoretic description of $+_\lambda$, without having to resort to heaps.

\begin{remark}\label{rem:sum_decomposition}
	Given $h,k\in\Hh^1$ and $a,b\in\Nn^1$ such that $ha+kb$ exists, one has 
	\[ ha+kb = (h+k)n \]
	where 
	\[n = \gamma_{h+k}^{-1} (\varphi_k\gamma_h(a) + \gamma_k(b)). \]
	The computations are immediate, and they are actually a subcase of the computations in the proof of \cref{thm:is_semidirect}.
	
	Observe that one can also write
	\begin{align*} n &= (h+k)^{-1}(ha+kb)\\
	\overset{(\dagger)}&{=} \big(h(h^{-1}\rightharpoonup k)\big)^{-1} (ha+kb)\\
	\overset{(\ddagger)}&{=} (h^{-1}\leftharpoonup k)k^{-1} (ha+kb) \\
	\overset{\eqref{qtbc}}&{=}(h^{-1}\leftharpoonup k)k^{-1} h a - (h^{-1}\leftharpoonup k)k^{-1} + (h^{-1}\leftharpoonup k) b, \end{align*}
	where the step marked with $(\dagger)$ follows from $h+k = h(h^{-1}\rightharpoonup k)$, see \cref{prop:braided_to_qtsb}, and the step marked with $(\ddagger)$ follows from $(h^{-1}\rightharpoonup k) = (h^{-1}\leftharpoonup k)k^{-1}h$, which in turn is an easy consequence of \eqref{eq:braidedcomm}. This alternative form for $n$ has an intuitive visual interpretation: defining \[x= a^{-1}h^{-1}\rightharpoonup k, \quad y = (a\leftharpoonup x)^{-1} \big( h\leftharpoonup (a\rightharpoonup x)\big)^{-1}\rightharpoonup b,\] one has $ha+kb = haxy$, and $(h+k)n$ is visualised as the diagonal of the grid in \cref{fig:tiles}.
\end{remark}
\begin{figure}[t]
	\begin{tikzpicture}[x=1.5cm, y=1.5cm,node font=\tiny]
		\draw[-Stealth] (0,0) to node[sloped, above]{$h$} (1,2);
		\draw[-Stealth] (1,2) to node[sloped, above]{$a$} (1.5,3);
		\draw[-Stealth] (0,0) to node[sloped, below]{$k$} (2,0);
		\draw[-Stealth] (2,0) to node[sloped, below]{$b$} (3,0);
		\draw[-Stealth,dotted] (0,0) to node[sloped, above]{$h+k$} (3,2);
		\draw[-Stealth] (2,0) to node[sloped, below]{$h\leftharpoonup(a\rightharpoonup x)$} (3,2);
		\draw[-Stealth] (1,2) to node[sloped, above]{$a\rightharpoonup x$} (3,2);
		\draw[-Stealth] (1.5,3) to node[sloped, above]{$x$} (3.5,3);
		\draw[-Stealth] (3.5,3) to node[sloped, above]{$y$} (4.5,3);
		\draw[-Stealth] (3,2) to node[sloped, above]{$a\leftharpoonup x$} (3.5,3);
		\draw[-Stealth] (3,2) to node[sloped, below]{$(a\leftharpoonup x)\rightharpoonup y$} (4,2);
		\draw[-Stealth,dotted] (3,2) to node[above, sloped]{$n$} (4.5,3);
		\draw[-Stealth] (4,2) to node[sloped, below]{$a\leftharpoonup xy$} (4.5,3);
		\draw[-Stealth] (3,0) to node[sloped, below]{$ (h\leftharpoonup (a\rightharpoonup x))\leftharpoonup ((a\leftharpoonup x)\rightharpoonup y) $} (4,2);
	\end{tikzpicture}
	\caption{Reference picture for \cref{rem:sum_decomposition}. Here each square tile represents an application of the braiding $r$, but the figure is not symmetric (unless $r$ is involutive): thus the tiles should be understood as oriented, with application of $r$ going from the top left to the bottom right.}\label{fig:tiles}
\end{figure}
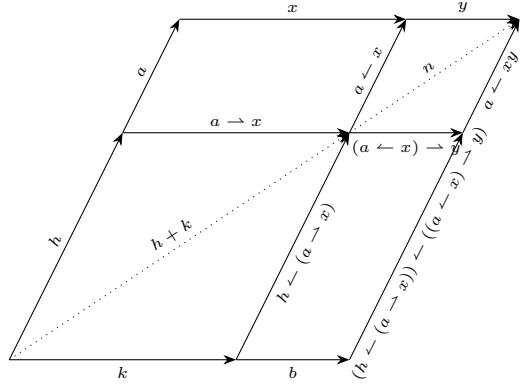
As a consequence of the above considerations, when $\Nn = \underline{\Nn}= \Nn^\circlearrowright$, we can describe the `bilateral' semidirect product $\Nn\rcross{}{}\Hh\lcross{}{}\Nn$ as a `unilateral' semidirect product.
\begin{definition}
	Consider the setting of \cref{def:crossed} where $\Nn = \underline{\Nn}= \Nn^\circlearrowright$. We define $\Hh\lcross{}{}\Nn$ as the quiver $\Hh\ot \Nn$ with the groupoid structure
	\[ (h\ot a)\cdot ( k\ot b) =  hk\ot a  (b\triangleleft h) \]
	(see \cite[Definition 5.9]{ferri2025splitandFIT}) and the additive structure
	\[   (h\ot a) + (k\ot b) = (h+k)\ot \gamma_{h+k}^{-1}(\varphi_k\gamma_h(a)+\gamma_k(b)).\]
\end{definition}
\begin{proposition}\label{prop:onesidedsemidirect}
	With the above structures, $\Hh\lcross{}{}\Nn$ is a quiver skew brace, and it is isomorphic to $\Nn\rcross{}{}\Hh\lcross{}{}\Nn$.
\end{proposition}
\begin{proof}
	The good definition follows immediately from \cref{lem:crossed_welldefined,prop:crossed_is_qskb} applied on triples $a\ot h\ot b$ with $a = 1_{\source(h)}$. If $\Nn$ is a bundle of loops, it is clear that every element $a\otb h\otb b$ in $\Nn\otb \Hh\otb \Nn$ can be rewritten as $1\otb h'\otb b'$ for unique $h'\in \Hh^1$, $b'\in \Nn^1$ (this is the same argument as in \cite[Lemma 5.10]{ferri2025splitandFIT}): this yields the desired isomorphism $\Nn\rcross{}{}\Hh\lcross{}{}\Nn\cong \Hh\lcross{}{}\Nn$.
\end{proof}
\begin{corollary}\label{cor:splitlemma_strong}
	Given a split strong epimorphism of quiver skew braces $\Gg\to \Hh$, with kernel $\Nn = \Nn^\circlearrowright$, one has $\Gg \cong \Hh\lcross{}{}\Nn$ as quiver skew braces.
\end{corollary}
\begin{proof}
	It is the conjunction of \cref{thm:is_semidirect,prop:onesidedsemidirect}.
\end{proof}

\section{Prunability}

\subsection{The main difference between $\maybebm{\Gpd}$ and $\maybebm{\QSKB}$}\label{sec:harder}  As it was observed e.g.\@ in \cite{GroupsToGroupoidsBrown}, every connected groupoid $\Gg$ fits in a short exact sequence
\[\begin{tikzcd}
	{\One_{\Gg^0}} & {\Gg^\circlearrowright} & \Gg & {\Gg/\Gg^\circlearrowright} & 1,
	\arrow[from=1-1, to=1-2]
	\arrow[from=1-2, to=1-3]
	\arrow[from=1-3, to=1-4]
	\arrow[from=1-4, to=1-5]
\end{tikzcd}\]
and choosing a maximal Schurian subgroupoid $\widehat{\Lambda}$ of $\Gg$ is equivalent to choosing a quiver-theoretic section $s\colon \Gg/\Gg^\circlearrowright \to \Gg$. Such a quiver-theoretic section can always be chosen so that it is also a morphism of groupoids. The choice, however, is generally not unique. By the Lifted Split Lemma in $\Gpd$ (see \cite{ferri2025splitandFIT}), this induces the decomposition $\Gg\cong \Gg^\circlearrowright \rcross{}{} (\Gg/\Gg^\circlearrowright)\lcross{}{}\Gg^\circlearrowright\cong \Gg^\circlearrowright \rcross{}{} (\Gg/\Gg^\circlearrowright)$.\footnote{At a first glance it seems like both terms of the decomposition $\Gg\cong \Gg^\circlearrowright \rcross{}{} (\Gg/\Gg^\circlearrowright)$ are independent of any choice. However, as remarked e.g.\@ in \cite{ferri2025splitandFIT}, this decomposition requires the choice of a maximal Schurian subgroupoid, and such an arbitrary choice cannot be avoided. This may look like a contradiction, but it is readily explained: it is the action $\triangleright$ of $\Gg/\Gg^\circlearrowright$ on $\Gg^\circlearrowright$, and hence the groupoid structure of the crossed product, that depends on how $\Gg/\Gg^\circlearrowright$ is immersed into $\Gg$. And the immersion of $\Gg/\Gg^\circlearrowright$ is precisely the choice of a maximal Schurian subgroupoid.} The same result in a slightly different flavour, applied to groupoid algebras, appears in \cite{DAdderioHautekietSaraccoVercruysse,DokuchaevExelPiccione,SteinbergSemigroupRepTheory,velasco2021thesis} and many other places.

We now switch to the category $\QSKB$. We have two `extremal' classes of quiver skew braces:
\begin{itemize}
	\item[(I)] skew braces, or more generally bundles of skew braces; and
	\item[(II)] coarse quiver skew braces (which are equivalent to heaps).
\end{itemize}%Every quiver skew brace lies somewhere between two extrema:\begin{itemize}	\item[(I)] bundles of skew braces; and 	\item[(II)] coarse quiver skew braces $\widehat{\Lambda}$.\end{itemize}A quiver skew brace structure on the coarse groupoid $\widehat{\Lambda}$ is the same as a group (or, better, a \textit{heap}) structure on the set of vertices $\Lambda$ (see \cite{FerriShibukawa} or \cite[Proposition 2.32]{ferri2024dynamical}). Thus an object of type (II) is `essentially' a group. On the other hand, an object of type (I) is a bundle of objects that, despite being less understood than groups, are still widely studied. For this reason, w
We would like to fit any connected quiver skew brace $\Gg$ into a short exact sequence as above, where $\Gg^\circlearrowright$ is an object of type (I), and $\Gg/\Gg^\circlearrowright$ is of type (II). This classification program fails immediately. The problem is that, while for a groupoid $\Gg$ the bundle of loops $\Gg^\circlearrowright$ is always a normal subgroupoid, for a quiver skew brace $\Gg$ it is generally false that $\Gg^\circlearrowright$ be an ideal.
\begin{example}\label{ex:unprunable}
	Consider the quiver skew brace $\Gg= \Kk_{2,3}$ from \cite[Example 4.29]{ferri2024dynamical}
	\[\begin{tikzcd}
		S_2 && S_3
		\arrow["0", from=1-1, to=1-1, loop, in=150, out=210, distance=5mm]
		\arrow["3", from=1-1, to=1-1, loop, in=140, out=220, distance=15mm]
		\arrow["2", bend left=10, from=1-1, to=1-3]
		\arrow["1", bend left=35, from=1-1, to=1-3]
		\arrow["2", bend left=10, from=1-3, to=1-1]
		\arrow["3", bend left=35, from=1-3, to=1-1]
		\arrow["0", from=1-3, to=1-3, loop, in=330, out=30, distance=5mm]
		\arrow["1", from=1-3, to=1-3, loop, in=320, out=40, distance=15mm]
	\end{tikzcd}\]
	where the unique arrow with source $S_i$ and label $a$ is denoted by $[S_i\Vert a]$, and the additive group structure on $\Gg(S_2, \Gg^0)$ is induced by the usual addition on $\Z/4\Z$:
	\[ [S_2\Vert a] +_{S_2} [S_2 a\Vert b] = [S_2\Vert a+b], \]
	and the groupoid structure is described in \cite[Table 3]{ferri2024dynamical} through left-quasigroup operations $\bullet_{S_2}, \bullet_{S_3}$. The left action $x\rightharpoonup y = -_{\source(x)}x+_{\source(x)}xy$, for $x = [S_2\Vert 2]$, sends the loop $[S_2\Vert 3]$ into the arrow
	\[[S_2\Vert 2]\rightharpoonup [S_2\Vert 3] = [S_2\Vert -2] + [S_2\Vert 2\bullet_{S_2} 3] = [S_2\Vert 1],   \] which is not a loop. Thus $\Gg^\circlearrowright$ is not invariant under the action $\rightharpoonup$, and hence it is not an ideal.
	
	Incidentally, this implies that $\Kk_{2,3}$ has no nontrivial ideals (we call such a quiver skew brace \textit{simple}). Thus $\Kk_{2,3}$ is the smallest simple quiver skew brace that is neither of type (I) nor of type (II).
\end{example} 
We set some terminology to talk about decompositions of quiver skew braces.
\begin{definition}
	A quiver skew brace $\Gg$ is \textit{prunable} if it has an ideal bundle $\Ii \neq \One_\Gg$. It is \textit{unprunable} otherwise. It is \textit{completely prunable} if $\Gg^\circlearrowright$ is an ideal.
\end{definition}

Clearly, simple quiver skew braces are unprunable. The converse does not hold. 
\begin{example}\label{ex:unprunable_nonsimple}
	Let $\Kk_{2,3}$ be as in \cref{ex:unprunable}, and $\Lambda = \{ 0,1\} = \Z/2\Z$ with its usual heap structure $\points*{a,b,c} = a-b+c$. We prove that $\Kk_{2,3} \times \widehat{\Lambda}$ is an unprunable quiver skew brace which is not simple. It is not simple because $\Ii= \Kk_{2,3}\times \One_{\widehat{\Lambda}}$ is a proper ideal by \cref{prop:ideal_of_prod}. Again by \cref{prop:ideal_of_prod}, this ideal $\Ii$ is also isomorphic to $\Kk_{2,3}\times \One_{\widehat{\Lambda}}$ as a quiver skew brace, and hence $\Ii^\circlearrowright = \Kk_{2,3}^\circlearrowright \sqcup \Kk_{2,3}^\circlearrowright$ cannot be an ideal, because $\Kk_{2,3}^\circlearrowright$ was not an ideal in $\Kk_{2,3}$ in the first place. Therefore, $\Kk_{2,3} \times \One_{\widehat{\Lambda}}$ is unprunable. Since  $(\Kk_{2,3} \times \widehat{\Lambda})^\circlearrowright = (\Kk_{2,3} \times \One_{\widehat{\Lambda}})^\circlearrowright$, one has \textit{a fortiori} that $\Kk_{2,3} \times \widehat{\Lambda}$ is unprunable.
\end{example}

Classifying unprunable or simple (finite) quiver skew braces, at the moment, seems far beyond our reach. The following problems seem more easily approachable, but they are better left to future investigation.
\begin{problem}
	Determine all possible underlying groupoids of the unprunable quiver skew braces.
\end{problem}
\begin{problem}
	Determine all possible underlying groupoids of the simple quiver skew braces.
\end{problem}

\subsection{Completely prunable quiver skew braces} Let $\Gg$ be connected and completely prunable. This means that $\Gg/\Gg^\circlearrowright$ is also a quiver skew brace,  and hence $\Gg^0$ has a heap structure $\points*{\blank,\blank,\blank}$.
\begin{lemma}\label{lem:prunablesplit}
	If $\Gg$ is completely prunable and connected, then the short exact sequence of quiver skew braces
	\[\begin{tikzcd}
		{\One_{\Gg^0}} & {\Gg^\circlearrowright} & \Gg & {\Gg/\Gg^\circlearrowright} & 1,
		\arrow[from=1-1, to=1-2]
		\arrow[from=1-2, to=1-3]
		\arrow[from=1-3, to=1-4]
		\arrow[from=1-4, to=1-5]
	\end{tikzcd}\]
	splits.
\end{lemma}
\begin{proof}
	Let $\zeta$ be a chosen vertex in $\Gg^0$. For each $a\in \Gg^0$, we choose an arrow $x_{\zeta, a}\colon \zeta\to a$. This choice can easily be extended to a section of groupoids $s\colon \Gg/\Gg^\circlearrowright\to \Gg$, as it is well known, by setting $s^1([\zeta, a])= x_{\zeta, a}$ and $s^1([a,b]) = x_{a,b} = [\zeta, a]^{-1}[\zeta, b]$. Let $\Uu$ be the coarse subgroupoid of $\Gg$ thus obtained.
	
	We define $x_{\zeta, a}+_\zeta x_{\zeta, b} = x_{\zeta, \points*{a,\zeta, b}}$. We now prove that there is a unique way of extending $+_\zeta$ to $\{ +_a\mid a\in \Gg^0\}$ so that $\Uu$ becomes a quiver skew brace. The request
	\begin{align*} x_{\zeta, a} (x_{a,b}+_a x_{a,c}) &= x_{\zeta, a} x_{a,b} -_\zeta x_{\zeta, a} +_\zeta x_{\zeta, a}x_{a,c} \\
		&= x_{\zeta, b} -_\zeta x_{\zeta, a} +_\zeta x_{\zeta, c}\\
		&= x_{\zeta, \points*{ b, \zeta , \points*{ \points*{\zeta, a, \zeta}, \zeta, c } }}\\
		&= x_{\zeta, \points*{b,a,c}} \end{align*}
	implies $x_{a,b}+_a x_{a,c} = x_{\zeta, a}^{-1}x_{\zeta, \points*{b,a,c}} = x_{a,\points*{b,a,c}}$. This is the only way to extend $+_\zeta$ to a quiver skew brace structure. 
	
	We now need to check that this extension is indeed a quiver skew brace structure. First, we observe that $+_a$ is a group operation on $\{ x_{a,b}\mid b\in \Gg^0 \}$: this follows immediately from the fact that the operation $\cdot_a$ on $\Gg^0$, given by $b\cdot_a c = \points*{b,a,c}$, is a group operation. Finally, we observe that the quiver skew brace compatibility holds everywhere, and not just when the starting vertex is $\zeta$: this verification is immediate, and analogous to the computations done above for $\zeta$.
\end{proof}
From now on, we fix a section $s$ and we simply identify $x_{a,b}$ with $[a,b]$.
\begin{remark}
	Fix $\zeta \in \Gg^0$, so that $\Gg^0$ becomes a group with $a\cdot b = \points*{a,\zeta, b}$; and let $G= \Gg_\zeta $. Every $x \in \St_\Gg(\zeta)$ can be written uniquely as $g [\zeta,a]$ for $g\in G$ and $a\in \Gg^0$: indeed for $x\colon \zeta \to a$ one has that $x [\zeta, a]^{-1} = g\in G$ is a loop on $\zeta$.
	
	Since $\Gg$ is connected and we have chosen a coarse subgroupoid $\Uu$, every isotropy group $\Gg_a$ is canonically isomorphic to $G$ via
	\[ \Gg_a\to \Gg_\zeta, \quad x\mapsto [\zeta, a] x [\zeta, a]^{-1}.  \]
	Thus, with a similar trick as above, every $x\in \Gg(a,b)$ is entirely determined by the datum of $[a,b]\in \Uu$ and $g\in G$, where $g =  [\zeta, a] x[a,b]^{-1} [\zeta, a]^{-1}=  [\zeta, a] x [b,\zeta] $.  This yields a bijection between $\Gg^1$ and $G\times \Gg^0 \times \Gg^0$.
\end{remark}
\begin{corollary}
	If $\Gg$ is a completely prunable quiver skew brace, then $\Gg$ is isomorphic to $ \Gg/\Gg^\circlearrowright \lcross{}{}\Gg^\circlearrowright$ as quiver skew braces.
\end{corollary}
\begin{proof}
	It follows from \cref{cor:splitlemma_strong} and the above considerations.
\end{proof}
As it happens for $\Gpd$, the isomorphism $\Gg\cong \Gg/\Gg^\circlearrowright \lcross{}{}\Gg^\circlearrowright$ is not canonical, and it depends on the choice of a section $\Gg/\Gg^\circlearrowright\to \Gg$, that is, of a maximal Schurian sub-quiver skew brace of $\Gg$.
\section*{Acknowledgements} 
The author is grateful to Ilaria Colazzo for her many remarks and questions, and for her insight on the content of this paper. Moreover, this work benefited from conversations with Alessandro Ardizzoni, Geoffrey Janssens, Isabel Martin-Lyons, Paolo Saracco, Hal Simpson, and Leandro Vendramin. 

The author was funded by the Università di Torino though a PNRR DM 118 scholarship; by the Vrije Universiteit Brussel through the bench fee OZR3762; and by Leandro Vendramin through his FWO Senior Research Project G004124N.

\begin{CJK}{UTF8}{gbsn}
\bibliographystyle{acm}
\bibliography{../../refs}
\end{CJK}
\end{document}